\documentclass[12pt]{article}

\DeclareMathAlphabet{\mathpzc}{OT1}{pzc}{m}{it}

\usepackage{amsmath}
\usepackage{amssymb}
\usepackage{amsfonts}
\usepackage{amsthm}
\usepackage{mathrsfs}
\usepackage{enumerate}
\usepackage[pdftex]{color, graphicx}

\newcommand{\cC}{{\mathcal C}}

\newcommand{\sL}{{\mathscr L}}

\newcommand{\pa}{\parallel}
\newcommand{\cd}{\cdot}
\newcommand{\sm}{\setminus}

\newcommand{\li}{l_\infty}
\newcommand{\la}{\lambda}

\newcommand{\C}{\mathcal{C}}

\newcommand{\Cp}{\mathcal{C}_P}
\newcommand{\N}{\mathcal{N}}
\newcommand{\Nh}{\mathcal{N}_H}
\newcommand{\Npp}{\mathcal{N}_{P'}}

\allowdisplaybreaks

\newcommand{\Remarks}{{\bf Remarks.} } 

\newtheorem{thm}{Theorem}[section]
\newtheorem{lem}[thm]{Lemma}
\newtheorem{defn}[thm]{Definition}
\newtheorem{prop}[thm]{Proposition}
\newtheorem{cor}[thm]{Corollary}

\begin{document}

\renewcommand{\thefootnote}{\arabic{footnote}}
 	
\title{Synthetic foundations of cevian geometry, III:\\The generalized orthocenter}

\author{\renewcommand{\thefootnote}{\arabic{footnote}}
Igor Minevich and Patrick Morton\footnotemark[1]}
\footnotetext[1]{The authors were partially supported by an Honors Research Fellowship from the IUPUI Honors Program.}
\maketitle

\begin{section}{Introduction.}
In Part II of this series of papers \cite{mm2} we studied the conic $\Cp=ABCPQ$ and its center $Z$, where $P$ is a point not on the extended sides of triangle $ABC$ or its anticomplementary triangle, and $Q=K(\iota(P))=K(P')$ is the complement of the isotomic conjugate of $P$ ({\it isotomcomplement}\footnote[2]{This is Grinberg's terminology.}, for short) with respect to $ABC$.  When $P=Ge$ is the Gergonne point of $ABC$, $\Cp$ is the Feuerbach hyperbola.  To prove this, we introduce in Section 2 a generalization of the orthocenter $H$ which varies with the point $P$.  \medskip

The generalized orthocenter $H$ of $P$ is defined to be the intersection of the lines through the vertices $A, B, C$, which are parallel, respectively, to the lines $QD, QE, QF$, where $DEF$ is the cevian triangle of $P$ with respect to $ABC$. We prove that the point $H$ always lies on the conic $\Cp$, as does the corresponding generalized orthocenter $H'$ for the point $P'$ (Theorem \ref{thm:HonCp}). Thus, the cevian conic $\Cp$ lies on the nine points
\[A, B, C, P, Q, P', Q', H, H',\]
where $Q'=K(\iota(P'))=K(P)$.  \medskip

In the first two parts \cite{mm1} and \cite{mm2} we used the affine maps $T_P, T_{P'}, \mathcal{S}_1=T_P \circ T_{P'}$, and $\lambda = T_{P'} \circ T_P^{-1}$, where $T_P$ is the unique affine map which takes $ABC$ to $DEF$ and $T_{P'}$ is defined similarly for the point $P'$.  (See Theorems 2.1 and 3.4 of \cite{mm2}.)  In Section 2 of this paper we prove the affine formula
\begin{equation}
\label{eqn:H}
H = K^{-1}T_{P'}^{-1}K(Q)
\end{equation}
for the point $H$ that we defined above and deduce that $H$ and $H'$ are related by $\eta(H) = H'$, where $\eta$ is the affine reflection we made use of in Part II \cite{mm2}. (See Theorem \ref{thm:K(Q)}.) The point $H$ is the anti-complement of the point
\begin{equation}
\label{eqn:O}
O = T_{P'}^{-1}K(Q),
\end{equation}
which is a generalization of the circumcenter. Several facts from Part I \cite{mm1}, including the Quadrilateral Half-turn Theorem, allow us to give a completely synthetic proof of the affine formulas (\ref{eqn:H}) and (\ref{eqn:O}).  We show that the circumscribed conic $\tilde{\C}_O$ of $ABC$ whose center is $O$ is the {\it nine-point conic} (with respect to the line at infinity $l_\infty$) for the quadrangle formed by the point $Q$ and the vertices of the anticevian triangle of $Q$ (for $ABC$).  Furthermore, the complement $K(\tilde{\C}_O)$ is the nine-point conic $\Nh$ of the quadrangle $ABCH$.  When $P=Ge$ is the Gergonne point, $Q=I$ is the incenter, $P'$ is the Nagel point, and (\ref{eqn:H}) and (\ref{eqn:O}) yield affine formulas for the orthocenter and circumcenter as functions of $I$.  \medskip

In Section 3 we study the relationship between the nine-point conic $\Nh$, the circumconic $\tilde{\C}_O$, and the inconic $\mathcal{I}$, which is the conic tangent to the sides of $ABC$ at the traces $D, E, F$ of the point $P$.  Its center is $Q$.  (See [mm1, Theorem 3.9].)  We also study the maps
\[\textsf{M}=T_P \circ K^{-1} \circ T_{P'} \ \ \textrm{and} \ \ \Phi_P=T_P \circ K^{-1} \circ T_{P'} \circ K^{-1}.\]
We show that these maps are homotheties or translations whose centers are the respective points
\[S=OQ \cdot GV = OQ \cdot O'Q' \ \ \textrm{and} \ \ Z = GV \cdot T_P(GV),\]
and use these facts to prove the Generalized Feuerbach Theorem, that $\Nh$ and $\mathcal{I}$ are tangent to each other at the point $Z$.  The proof boils down to the verification that $\Phi_P$ takes $\Nh$ to $\mathcal{I}$, leaving $Z$ and its tangent line to $\Nh$ invariant.  Thus, this proof continues the theme, begun in Part I \cite{mm1}, of characterizing important points as fixed points of explicitly given affine mappings. \medskip

When $\Nh$ is an ellipse, the fact that $\Nh$ and $\mathcal{I}$ are tangent could be proved by mapping $\tilde{\C}_O$ to the circumcircle and $ABC$ to a triangle $A'B'C'$ inscribed in the same circumcircle; and then using the original Feuerbach theorem for $A'B'C'$.  The proof we give does not assume Feuerbach's original theorem, and displays explicitly the intricate affine relationships holding between the various points, lines, and conics that arise in the proof.  (See Figure 2 in Section 3.)  It also applies when $\Nh$ is a parabola or a hyperbola, and when $Z$ is infinite.  (See Figures 3 and 4 in Section 3.  Also see [mo], where a similar proof is used to prove Feuerbach's Theorem in general Hilbert planes.)  \medskip

In Section 4 we determine synthetically the locus of points $P$ for which the generalized orthocenter is a vertex of $ABC$.  This turns out to be the union of three conics minus six points.  (See Figure 5 in Section 4.)  We also consider a special case in which the map $\textsf{M}$ is a translation, so that the circumconic $\tilde{\C}_O$ and the inconic are congruent.  (See Figures 6 and 7.) \medskip

The results of this paper, as well as those in Part IV, could be viewed as results relating to a generalized notion of perpendicularity. The inconic replaces the incircle, and the lines $QD, QE, QF$ replace the perpendicular segments connecting the incenter to the points of contact of the incircle with the sides of $ABC$.  In this way we obtain the generalized incenter $Q$, generalized orthocenter $H$, generalized circumcenter $O$, generalized nine-point center $N$, etc., all of which vary with the point $P$.  Using the polarity induced by the inconic, we have an involution of conjugate points on the line at infinity, unless $P$ lies on the Steiner circumellipse.   Now Coxeter's well-known development (see \cite{co1}, Chapter 9 and \cite{bach}) of Euclidean geometry from affine geometry makes use of an elliptic involution on the line at infinity (i.e., a projectivity $\psi$ from $l_\infty$ to itself without fixed points, such that $\psi(\psi(X)) = X$ for all $X$). This involution is just the involution of perpendicularity: the direction perpendicular to the direction represented by the point $X$ at infinity is represented by $\psi(X)$.  \medskip

Our development is, however, not equivalent to Coxeter's. If $P$ lies inside the Steiner circumellipse, then the inconic is an ellipse and the
involution is elliptic, but if $P$ lies outside the Steiner circumellipse, the inconic is a hyperbola and the involution is hyperbolic. Furthermore, if
$P$ is on the Steiner circumellipse, then $P' = Q = H = O$ is at infinity, the inconic is a parabola, and there is no corresponding involution on the line at infinity.  However, interesting theorems of Euclidean geometry can be proved even in the latter two settings, which cannot be derived by applying an affine map to the standard results, since an affine map will take an elliptic involution on $l_\infty$ to another elliptic involution.

\end{section}

\begin{section}{Affine relationships between $Q, O$, and $H$.}

We continue to consider the affine situation in which $Q$ is the isotomcomplement of $P$ with respect to triangle $ABC$, and $DEF$ is the cevian triangle for $P$ with respect to $ABC$, so that $D=AP \cdot BC$, $E=BP \cdot AC$, and $F=CP \cdot AB$.  As in Parts I and II, $D_0E_0F_0=K(ABC)$ is the medial triangle of $ABC$.

\begin{defn}
The point $O$ for which $OD_0 \pa QD, OE_0 \pa QE$, and $OF_0 \pa QF$ is called the {\bf generalized circumcenter} of the point $P$ with respect to $ABC$.  The point $H$ for which $HA \pa QD, HB \pa QE$, and $HC \pa QF$ is called the {\bf generalized orthocenter} of $P$ with respect to $ABC$.
\end{defn}

We first prove the following affine relationships between $Q, O$, and $H$.

\begin{thm}
\label{thm:HO}
The generalized circumcenter $O$ and generalized orthocenter $H$ exist for any point $P$ not on the extended sides or the anticomplementary triangle of $ABC$, and are given by
$$O=T_{P'}^{-1}K(Q), \ \ H = K^{-1}T_{P'}^{-1}K(Q).$$
\end{thm}
 \noindent {\bf Remark.} The formula for the point $H$ can also be written as $H=T_L^{-1}(Q)$, where $L=K^{-1}(P')$ and $T_L$ is the map $T_P$ defined for $P=L$ and the anticomplementary triangle of $ABC$. \smallskip
 
\begin{proof}
We will show that the point $\tilde O=T_{P'}^{-1}K(Q)$ satisfies the definition of $O$, namely, that
$$\tilde OD_0 \pa QD, \ \ \tilde OE_0 \pa QE, \ \ \tilde OF_0 \pa QF.$$
It suffices to prove the first relation $\tilde OD_0 \pa QD$.  We have that
$$T_{P'}(\tilde OD_0)=K(Q)T_{P'}(D_0)=K(Q)A_0'$$
and
$$T_{P'}(QD)=P'A_3',$$
by I, Theorem 3.7.  Thus, we just need to prove that $K(Q)A_0' \pa P'A_3'$.  We use the map $\mathcal{S}_2=T_{P'}T_P$ from I, Theorem 3.8, which takes $ABC$ to $A_3'B_3'C_3'$.  We have $\mathcal{S}_2(Q)=T_{P'}T_P(Q)=T_{P'}(Q)=P'$.  Since $\mathcal{S}_2$ is a homothety or translation, this gives that $AQ \pa \mathcal{S}_2(AQ)=A_3'P'$.  Now note that $M'=K(Q)$ in I, Corollary 2.6, so
$$K(Q)A_0'=M'A_0'=D_0A_0',$$
by that corollary.  Now the Quadrilateral Half-turn Theorem (I, Theorem 2.5) implies that $AQ \pa D_0A_0'$, and therefore $P'A_3' \pa K(Q)A_0'$.  This proves the formula for $O$.  To get the formula for $H$, just note that $K^{-1}(OD_0)=K^{-1}(O)A, K^{-1}(OE_0)=K^{-1}(O)B, K^{-1}(OF_0)=K^{-1}(O)C$ are parallel, respectively, to $QD,QE,QF$, since $K$ is a dilatation.  This shows that $K^{-1}(O)$ satisfies the definition of the point $H$.
\end{proof}
\begin{cor} The points $O$ and $H$ are ordinary whenever $P$ is ordinary and does not lie on the Steiner circumellipse $\iota(l_\infty)$.  If $P$ does lie on $\iota(l_\infty)$, then $O=H=Q$.
\end{cor}
\begin{proof}
If $P$ lies on $\iota(l_\infty)$, then $P'=Q$ is infinite, and since $K$ is a dilatation, we have that $O=H=T_{P'}^{-1}(Q)=T_{P'}^{-1} T_P^{-1}(Q)=\mathcal{S}_1^{-1}(Q)=Q$ by I, Theorems 3.2 and 3.8.
\end{proof}

To better understand the connection between the point $P$ and the points $O$ and $H$ of Theorem \ref{thm:HO} we consider the circumconic $\tilde{\C}_O$ on $ABC$ with center $O$. We will show that this circumconic is related to the nine-point conic $\Npp$ (with respect to $l_\infty$) on the quadrangle $ABCP'$. Recall that this is the conic through the diagonal points and the midpoints of the sides of the quadrangle [co1, p. 84]; alternatively, it is the locus of the centers of conics which lie on the vertices of the quadrangle $ABCP'$. Three of these centers are the points
$$D_3 = AP' \cdot BC, \ \ E_3 = BP' \cdot AC, \ \ F_3 = CP' \cdot AB,$$
which are centers of the degenerate conics
$$AP' \cup BC, \ \ BP' \cup AC, \ \ CP' \cup AB.$$
Since the quadrangle is inscribed in $\Cp$, its center $Z$ lies on $\Npp$.

\begin{figure}
\[\includegraphics[width=5.5in]{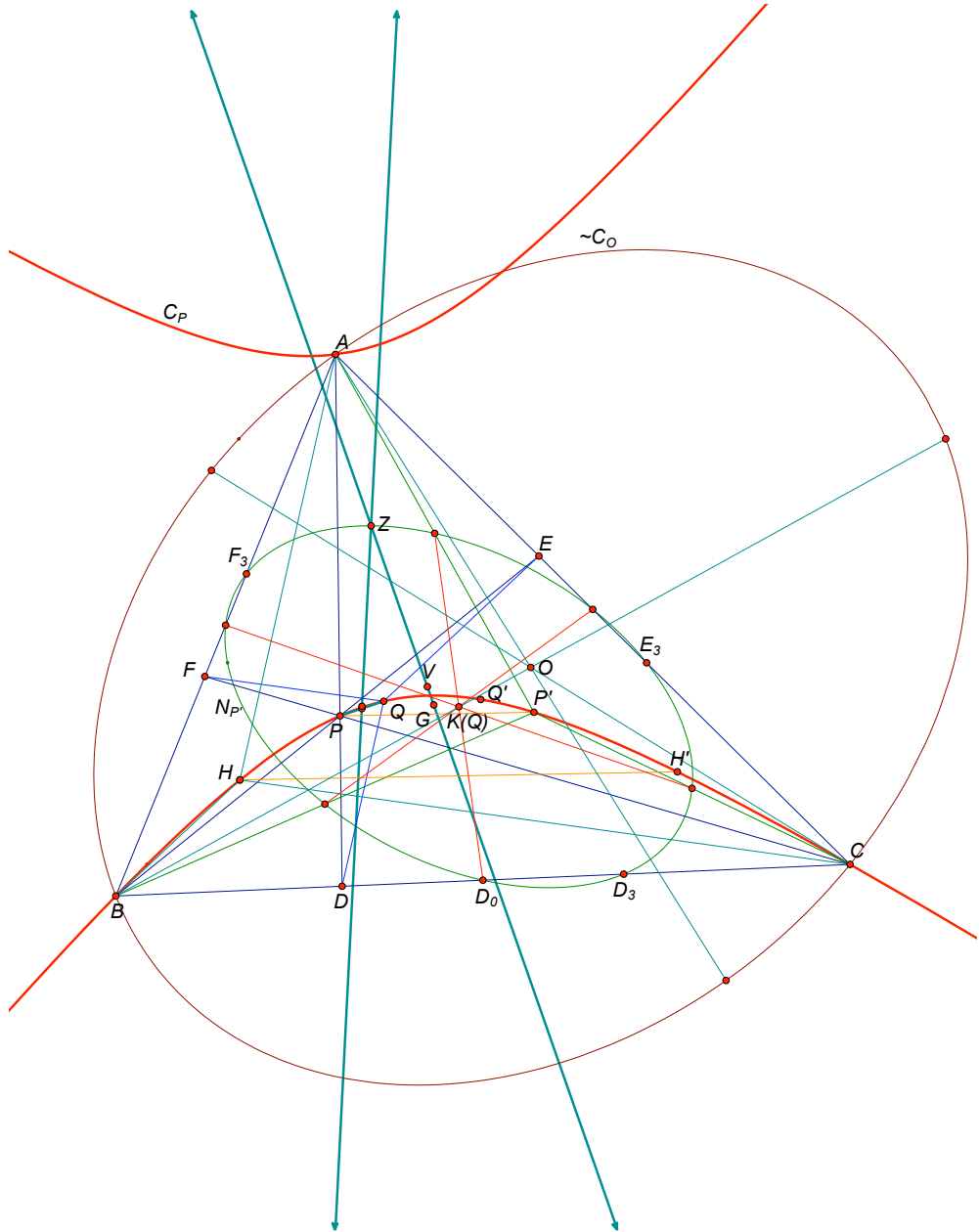}\]
\caption{Circumconic $\tilde{\mathcal{C}}_O$ and Nine-point Conic $\mathcal{N}_{P'}$}
\label{fig:2.1}
\end{figure}

\begin{thm}
\label{thm:K(Q)}
\begin{enumerate}[a)]
\item The point $K(Q)$ is the center of the nine-point conic $\Npp$ (with respect to $l_\infty$) for the quadrangle $ABCP'$.
\item The circumconic $\tilde{\C}_O = T_{P'}^{-1}(\Npp)$ is the nine-point conic for the quadrangle formed by the point $Q$ and the vertices of the anticevian triangle of $Q$. Its center is $O = T_{P'}^{-1}K(Q)$.
\item If $P$ does not lie on a median or on $\iota(l_\infty)$ and $O', H'$ are the points of Theorem 2.2 corresponding to the point $P'$, then $O' = \eta(O)$ and $H' = \eta(H)$. Thus $OO'$ and $HH'$ are parallel to $PP'$.
\end{enumerate}
\end{thm}

\begin{proof}
a) If $P'$ is ordinary, the conic $\Npp$ lies on the midpoints $D_0, E_0, F_0$, the vertices $D_3, E_3, F_3$, and the midpoints $R'_a, R'_b, R'_c$ of $AP', BP'$, and $CP'$. By I, Corollary 2.6, the point $K(Q)$ is the midpoint of segments $D_0R'_a, E_0R'_b, F_0R'_c$, which are chords for the conic $\Npp$. As the midpoint of two intersecting chords, $K(Q)$ is conjugate to two different points at infinity (with respect to the polarity associated with $\Npp$), and so must be the pole of the line at infinity. If $P'$ is infinite, the nine-point conic $\Npp$ lies on the midpoints $D_0, E_0, F_0$, the vertices $D_3, E_3, F_3$, and the harmonic conjugate of $AQ \cdot l_\infty=Q$ with respect to $A$ and $Q$, which is $Q$ itself. (See [co1, 6.83, p. 84].)  In this case we claim that $Q=K(Q)$ is the pole of $l_\infty$.  The line $l_\infty$ intersects $\Npp$ at least once at $Q$; it is enough to show that $l_\infty$ intersects $\Npp$ {\it only} at $Q$, because that will imply it is the tangent at $Q$, hence its pole is $Q$. Suppose $l_\infty$ also intersects $\Npp$ at a point $X$. The nine-point conic $\Npp$ is the locus of centers (that is, poles of $l_\infty$) of conics that lie on $A, B, C$, and $P'=Q$, which means $X$ is the pole of $l_\infty$ with respect to one of these conics $\cC$. Now $X$ cannot be $D_3, E_3$, or $F_3$, which are the centers of the degenerate conics through $ABCQ$, because none of these points lies on $l_\infty$.  Thus $\cC$ is nondegenerate.  By assumption, $X$ lies on its polar $l_\infty$ with respect to $\cC$, so $X$ lies on $\cC$ and $l_\infty$ is the tangent line at $X$. But we assumed that $Q$ and $X$ are distinct points on $l_\infty$, so $l_\infty$ is also a secant of $\cC$, a contradiction.  \smallskip

b) By I, Corollary 3.11, the anticevian triangle $A'B'C' $ of $Q$ is  $T_{P'}^{-1}(ABC)$. By I, Theorem 3.7, $Q = T_{P'}^{-1}(P')$, so the quadrangle $ABCP'$ is mapped to quadrangle $A'B'C'Q$ by the map $T_{P'}^{-1}$. The diagonal points $D_3, E_3, F_3$ of quadrangle $ABCP'$ map to $A, B, C$ so $T_{P'}^{-1}(\Npp)$ is certainly a circumconic for triangle $ABC$ with center $T_{P'}^{-1}K(Q)=O$, by Theorem \ref{thm:HO}. \smallskip

c) By Theorem \ref{thm:HO}, II, Theorem 2.4, and the fact that the map $\eta$ (see the discussion following Prop. 2.3 in Part II) commutes with the complement map, we have that
\[\eta(O) = \eta T_{P'}^{-1}K(Q) = T_{P}^{-1}K(\eta(Q)) = T_{P}^{-1}K(Q') = O',\]
and similarly $\eta(H) = H'$.
\end{proof}

We show now that there are 4 points $P$ which give rise to the same generalized circumcenter $O$ and generalized orthocenter $H$. These points arise in the following way. Let the vertices of the anticevian triangle for $Q$ with respect to $ABC$ be denoted by
\[Q_a = T_{P'}^{-1}(A), Q_b = T_{P'}^{-1}(B), Q_c = T_{P'}^{-1}(C).\]
Then we have
\[A = Q_bQ_c \cd QQ_a, B = Q_aQ_c\cd QQ_b, C = Q_aQ_b \cd QQ_c.\]
This clearly implies that $QQ_aQ_b$ is the anticevian triangle of $Q_c$ with respect to $ABC$. Similarly, the anticevian triangle of any one of these four points is the triangle formed by the other three (analogous to the corresponding property of the excentral triangle). We let $P_a, P_b, P_c$ be the anti-isotomcomplements of the points $Q_a, Q_b, Q_c$ with respect to $ABC$, so that
\[P_a = \iota^{-1}K^{-1}(Q_a), \ P_b = \iota^{-1}K^{-1}(Q_b),  \ P_c = \iota^{-1}K^{-1}(Q_c).\]

\begin{thm}
\label{thm:Pabc}
The points $P, P_a, P_b, P_c$ all give rise to the same generalized circumcenter $O$ and generalized orthocenter $H$.
\end{thm}

\begin{proof}
We use the characterization of the circumconic $\tilde{\C}_O$ from Theorem \ref{thm:K(Q)}(b). It is the nine-point conic $\N_Q$ for the quadrangle $Q_aQ_bQ_cQ$. But this quadrangle is the same for all of the points $P, P_a, P_b, P_c$, by the above discussion, so each of these points gives rise to the same circumconic $\tilde{\C}_O$. This implies the theorem, since $O$ is the center of $\tilde{\C}_O$ and $H$ is the anti-complement of $O$.
\end{proof}

\begin{cor}
The point-set $\{A, B, C, H\}$ is the common intersection of the four conics $\cC_Y$, for $Y = P, P_a, P_b, P_c$.
\end{cor}

\begin{proof}
If $P$ does not lie on a median of $ABC$, the points $Q, Q_a, Q_b, Q_c$ are all distinct, since $Q_aQ_bQ_c$ is the anticevian triangle of $Q$ with respect to $ABC$.  It follows that the points $P, P_a, P_b, P_c$ are distinct, as well.  If two of the conics $\mathcal{C}_Y$ were equal, say $\mathcal{C}_{P_a}=\mathcal{C}_{P_b}$, then this conic lies on the points $Q_a, Q_b$ and $QQ_c \cdot Q_aQ_b =C$, which is impossible.  This shows that the conics $\mathcal{C}_Y$ are distinct.
\end{proof}

The points $Q_a, Q_b, Q_c$ are the analogues of the excenters of a triangle, and the points $P_a, P_b, P_c$ are the analogues of the external Gergonne points. The traces of the points $P_a, P_b, P_c$ can be determined directly from the definition of $H$; for example, $Q_cD_c, Q_cE_c, Q_cF_c$ are parallel to $AH, BH, CH$, which are in turn parallel to $QD, QE, QF$.  \medskip

The next theorem shows that the points $H$ and $H'$ are a natural part of the family of points that includes $P, Q, P', Q'$.

\begin{thm}
\label{thm:lambda}
If the ordinary point $P$ does not lie on a median of $ABC$ or on $\iota(\li)$, we have:
\begin{enumerate}[a)]
\item $\la(P) = Q'$ and $\la^{-1}(P') = Q$.
\item $\la(H) = Q$ and $\la^{-1}(H') = Q'$.
\end{enumerate}
\end{thm}

\noindent \Remarks This gives an alternate formula for the point $H = \la^{-1}(Q) = \eta \la^2(P)$. Part b) gives an alternate proof that $\eta(H) = H'$.

\begin{proof}
Part a) was already proved as part of the proof of II, Theorem 3.2. Consider part b), which we will prove by showing that
\begin{equation}
\label{eqn:4.2}
T_P^{-1}(H) = T_{P'}^{-1}(Q).
\end{equation}
The point $T_P^{-1}(H)$ is the generalized orthocenter for the point $T_P^{-1}(P) = Q'$ and the triangle $T_P^{-1}(ABC) = \tilde A \tilde B \tilde C$, which is the anticevian triangle for $Q'$ with respect to $ABC$ (I, Corollary 3.11(a)). It follows that the lines $T_P^{-1}(H)\tilde A$ and $QA$ are parallel, since $Q$ is the isotomcomplement of $Q'$ with respect to $\tilde A \tilde B \tilde C$ (I, Theorem 3.13), and $A$ is the trace of the cevian $\tilde A Q'$ on $\tilde B \tilde C$. We will show that the line $T_{P'}^{-1}(Q)\tilde A$ is also parallel to $QA$. For this, first recall from I, Theorem 2.5 that $QD_0$ and $AP'$ are parallel.  Part II,  Theorem 3.4(a) implies that $Q, D_0$, and $\la(A)$ are collinear, so it follows that $Q\la(A)$ and $D_3P' = AP'$ are parallel. Now apply the map $T_{P'}^{-1}$. This shows that $T_{P'}^{-1}(Q\la(A)) = T_{P'}^{-1}(Q)T_{P'}^{-1}\la(A) = T_{P'}^{-1}(Q)\tilde A$ is parallel to $T_{P'}^{-1}(D_3P') = AQ$, as claimed. Applying the same argument to the other vertices of $T_P^{-1}(ABC) = \tilde A \tilde B \tilde C$ gives (\ref{eqn:4.2}), which is equivalent to $\la(H) = Q$. Now apply the map $\eta$ to give $\la^{-1}(H') = Q'$.
\end{proof}

\begin{thm}
\label{thm:HonCp}
If $P$ does not lie on a median of triangle $ABC$ or on $\iota(l_\infty)$, the generalized orthocenter $H$ of $P$ and the generalized orthocenter $H'$ of $P'$ both lie on the cevian conic $\mathcal{C}_P=ABCPQ$.
\end{thm}
\begin{proof}
Theorem \ref{thm:lambda}(b) and the fact that $\lambda$ maps the conic $\mathcal{C}_P$ to itself (II, Theorem 3.2) imply that $H$ lies on $\mathcal{C}_P$, and  since $Q'$ lies on $\mathcal{C}_P$, so does $H'$.
\end{proof}

\begin{cor}
When $P$ is the Gergonne point of $ABC$, the conic $\Cp = ABCPQ$ is the Feuerbach hyperbola $ABCHI$ on the orthocenter $H$ and the incenter $Q=I$.  The Feuerbach hyperbola also passes through the Nagel point ($P'$), the Mittenpunkt ($Q'$), and the generalized orthocenter $H'$ of the Nagel point, which is the point of intersection of the lines through $A, B, C$ which are parallel to the respective lines joining the Mittenpunkt to the opposite points of external tangency of $ABC$ with its excircles.
\end{cor}

Part II, Theorem 3.4 gives six other points lying on the Feuerbach hyperbola. Among these is the point $A_0P \cd D_0Q'$, where $D_0Q'$ is the line joining the Mittenpunkt to the midpoint of side $BC$, and $A_0P$ is the line joining the Gergonne point to the opposite midpoint of its cevian triangle. Using I, Theorem 2.4 and the fact that the isotomcomplement of the incenter $Q=I$ with respect to the excentral triangle $I_aI_bI_c$ is the Mittenpunkt $Q'$ (see I, Theorem 3.13), it can be shown that the line $D_0Q'$ also passes through the excenter $I_a$ of triangle $ABC$ which lies on the angle bisector of the angle at $A$. Thus, $A_0P \cd D_0Q' = A_0P \cd I_aQ'$. By II, Theorem 3.4(b), the lines $DQ$ and $A'_3P'$ also lie on this point.  \medskip

From (\ref{eqn:4.2}) we deduce

\begin{thm}
\label{thm:PerspectivitesWithH}
Assume that the ordinary point $P$ does not lie on a median of $ABC$ or on $\iota(\li)$.
\begin{enumerate}[a)]
\item The point $Q$ is the perspector of triangles $D_0E_0F_0$ and $\la(ABC)$.
\item The point $\tilde H = T_P^{-1}(H) = T_{P'}^{-1}(Q)$ is the perspector of the anticevian triangle for $Q'$ (with respect to $ABC$) and the medial triangle of the anticevian triangle of $Q$. Thus, the points $\tilde A = T_P^{-1}(A), \tilde H$, and $T_{P'}^{-1}(D_0)$ are collinear, with similar statements for the other vertices.
\item The point $H$ is the perspector of triangle $ABC$ and the medial triangle of $\la^{-1}(ABC)$. Thus, $A, H$, and $\la^{-1}(D_0)$ are collinear, with similar statements for the other vertices.
\item The point $\tilde H$ is also the perspector of the anticevian triangle of $Q$ and the cevian triangle of $P'$. Thus, $\tilde H$ is the $P'$-ceva conjugate of $Q$, with respect to triangle $ABC$.
\item The point $H$ is also the perspector of $\la^{-1}(ABC)$ and the triangle $A_3B_3C_3$.
\end{enumerate}
\end{thm}

\begin{proof}
a) As in the previous proof, the points $Q, D_0$, and $\la(A)$ are collinear, with similar statements for $Q, E_0, \la(B)$ and for $Q, F_0, \la(C)$. Therefore, $Q$ is the perspector of triangles $D_0E_0F_0$ and $\la(ABC)$. \smallskip

b) Now apply the map $T_{P'}^{-1}$, giving that $T_{P'}^{-1}(Q)$ is the perspector of $T_{P'}^{-1}(D_0E_0F_0)$, which is the medial triangle of $T_{P'}^{-1}(ABC)$, and the triangle $T_{P'}^{-1}\la(ABC) = T_P^{-1}(ABC)$. The result follows from (\ref{eqn:4.2}) and I, Corollary 3.11.

c) This part follows immediately by applying the map $T_P$ to part (b) or $\lambda^{-1}$ to part (a). \smallskip

d) I, Theorem 3.5 gives that $A, A'_1$ (on $E_3F_3$), and $Q$ are collinear, from which we get that $T_{P'}^{-1}(A) = A', T_{P'}^{-1}(A'_1) = D'_1 = D_3$, and $T_{P'}^{-1}(Q) = \tilde H$ are collinear. This and the corresponding statements for the other vertices imply (d).  \smallskip

e) Applying the map $T_P$ to (d) show that $\la^{-1}(A), A_3$, and $H$ are collinear, with similar statements for the other vertices.
\end{proof}

\end{section}

\begin{section}{The generalized Feuerbach Theorem.}

Recall that the $9$-point conic $\Nh$, with respect to the line $l_\infty$ is the conic which passes through the following nine points: \medskip

the midpoints $D_0, E_0, F_0$ of the sides $BC, AC, AB$;

the midpoints $R_1, R_2, R_3$ of the segments $AH, BH, CH$;

the traces (diagonal points) $H_1,H_2,H_3$ of $H$ on the sides $BC,AC,AB$. \medskip

\noindent In this section we give a synthetic proof, based on the results of this paper, that $\Nh$ is tangent to the inconic of $ABC$ whose center is $Q$, at the generalized Feuerbach point $Z$.  We start by showing that the conics $\Nh$ and $\tilde{\C}_O$ are in the same relationship to each other as the classical $9$-point circle and circumcircle.

\begin{thm}
\label{thm:NH}
Assume that $P$ is a point for which the generalized circumcenter $O$ is not the midpoint of any side of $ABC$, so that $H$ does not coincide with any of the vertices $A, B$, or $C$.  As in Theorem 2.4, $\tilde{\C}_O$ is the unique circumconic of $ABC$ with center $O$.  Then the $9$-point conic $\Nh$ is the complement of $\tilde{\C}_O$ with respect to $ABC$.  It is also the image of $\tilde{\C}_O$ under a homothety $\textsf{H}$ centered at $H$ with factor $1/2$.  Its center is the midpoint $N$ of segment $HO$.
\end{thm}

\begin{proof}
Let $T_1,T_2,T_3$ denote the reflections of $A,B,C$ through the point $O$, and denote by $S_1,S_2,S_3$ the intersections of $AH,BH,CH$ with the conic $\C=\tilde{\C}_O$.  Let $\textsf{R}_O$ the mapping which is reflection in $O$.  By Theorem \ref{thm:HO}, $K(H)=O$, so $K(O)=N$.  Hence $K(\C)$ is a conic through the midpoints of the sides of $ABC$ with center $N$.  Reflecting the line $AH$ in the point $O$, we obtain the parallel line $T_1S_1'$ with $S_1'=\textsf{R}_O(S_1)$ on $\C$.  This line contains the point $\bar{H}=\textsf{R}_O(H)=K^{-1}(H)$, since the centroid $G$ lies on $OH$.  If we define $R_1=T_1G \cdot AH$, then triangle $HGR_1$ is similar to $\bar{H}GT_1$.  Hence, $R_1=K(T_1)$ lies on $K(\C)$ and $R_1H=K\textsf{R}_O(AH)$, so $R_1$ is the midpoint of $AH$.  In the same way, $R_2=K(T_2)$ and $R_3=K(T_3)$ lie on the conic $K(\C)$.  This shows that $K(\C)$ lies on the midpoints of the sides of the quadrangle $ABCH$, and so is identical with the conic $\Nh$.  Since the affine map $\textsf{H}=K\textsf{R}_O$ takes $A,B,$ and $C$ to the respective midpoints of $AH, BH, CH$ and fixes $H$, it is clear that $\textsf{H}$ is the homothety centered at $H$ with factor $1/2$, and $\textsf{H}(\C)=K\textsf{R}_O(\C)=K(\C)=\Nh$.
\end{proof}

\begin{figure}
\[\includegraphics[width=5.5in]{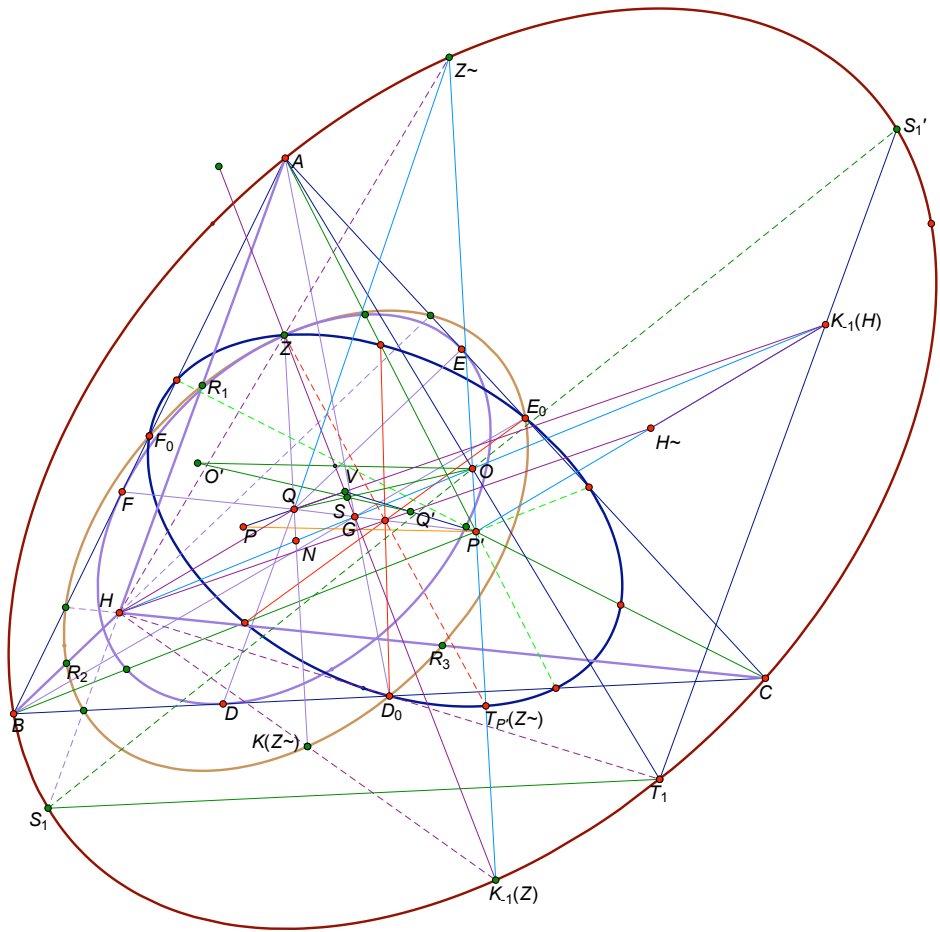}\]
\caption{$\Npp, \tilde{\mathcal{C}}_O$, $\mathcal{N}_{H}$, and $\mathcal{I}$ \ ($K_{-1}=K^{-1}$)}
\label{fig:3.1}
\end{figure}

\begin{prop}
\label{thm:STparBC}
Assume the hypothesis of Theorem \ref{thm:NH}.  If $O$ is not on the line $AH$, and the intersections of the cevians $AH$ and $AO$ with the circumconic $\tilde{\C}_O$ are $S_1$ and $T_1$, respectively, then $S_1T_1$ is parallel to $BC$.
\end{prop}

\begin{proof}
(See Figure 2.) First note that $P$ does not lie on the median $AG$, since $H$ and $O$ lie on this median whenever $P$ does, by Theorem \ref{thm:K(Q)} and the arguments in the proof of II, Theorem 2.7, according to which $T_P$ and $T_{P'}$ fix the line $AG$ when it lies on $P$.  Thus, $H$ cannot lie on any of the secant lines $AB, AC$, or $BC$ of the conic $\Cp=ABCPQ$, because $H$ is also on this conic and does not coincide with a vertex.  It is also easy to see that $H$ does not lie on a side of the anticomplementary triangle $K^{-1}(ABC)$, since $Q$ does not lie on a side of $ABC$.  By the proof of I, Theorem 2.5 (Quadrilateral Half-turn Theorem), with $H$ as the point $P$ and $O=K(H)$ as the point $Q'$, we know that $R_1$ is the reflection of $O$ in the midpoint $N_1$ of $E_0F_0$.  Since $R_1$ is the midpoint of $AH$, the line $R_1O$ is a midline of triangle $AHT_1$; and since $N_1$ is the midpoint of segment $R_1O$, we know that $AN_1$ intersects the opposite side $HT_1$ in its midpoint $M$, with $N_1$ halfway between $A$ and $M$.
However, $AN_1 \cong N_1D_0$, because $N_1$ is the midpoint of the midline $E_0F_0$ of triangle $ABC$.  Therefore, $D_0=M$ is the midpoint of $HT_1$.  Now, considering triangle $HS_1T_1$, with $OD_0 \pa HS_1$ (on $AH$), we know that $OD_0$ intersects $S_1T_1$ in its midpoint.  But $D_0$ is the midpoint of the chord $BC$ of the conic $\tilde{\C}_O$, and $S_1T_1$ is also a chord of $\tilde{\C}_O$.  Since $O$ is the center of $\tilde{\C}_O$, it follows that $OD_0$ is the polar of the points at infinity of both $BC$ and $S_1T_1$ (with respect to $\tilde{\C}_O$), whence these lines must be parallel.
\end{proof}

When the point $H$ does coincide with a vertex, we define $\Nh$ by $\Nh=K(\tilde{\C}_O)$.

\begin{prop}
If the generalized orthocenter $H=A$, then $\Nh=K(\tilde{\C}_O)$ is the unique conic on the vertices of the quadrangle $AD_0E_0F_0$ which is tangent to the conic $\tilde{\C}_O=T_{P'}^{-1}(\mathcal{N}_{P'})$ at $A$.
\end{prop}

\begin{proof}
If $H = A$, then $O = D_0$ is the center of $\tilde C_O$, so the reflection of $A$ through $D_0$ lies on $\tilde C_O$, and this reflection is precisely the anticomplement $A'=K^{-1}(A)=K^{-2}(D_0)$ of $A$.  Then clearly $A=K(A')$ lies on $\Nh=K(\tilde{\C}_O)$. Thus $\Nh$ is a conic on the vertices of $AD_0E_0F_0$.  The tangent line $t'$ to $\tilde{\C}_O$ at $A'$ is taken by the complement map to a parallel tangent line $t_1$ to $\Nh$ at $A$.  But the tangent line $t_2$ to $\tilde{\C}_O$ at $A$ is also parallel to $t'$ because the line $AA'$ lies on the center $O$ (dually, $a \cdot a'$ lies on $l_\infty$).  It follows that $t_1=t_2$, hence $\Nh$ is tangent to $\tilde{\C}_O$ at $A$.
\end{proof}

In this case, we take the point $S_1$ in the above proposition to be the intersection of the conic $\tilde{\C}_O$ with the line $t_A$ through $A$ which is parallel to $QD$, and we claim that we still have $S_1T_1 \pa BC$.  To prove this we first prove the following theorem.

\begin{thm}
\label{thm: CtoI}
The map $\textsf{M}=T_PK^{-1}T_{P'}$ is a homothety or translation taking the conic $\tilde{\C}_O$ to the inconic $\mathcal{I}$, which is the conic with center $Q$ tangent to the sides of $ABC$ at the points $D,E,F$.
\end{thm}

\noindent {\bf Remark.} Below we will show that the fixed point (center) of the map $\textsf{M}$ is the point $S=OQ \cdot GV$, which coincides with the point $\gamma_P(P)=Q \cdot Q'$, where $\gamma_P$ is the generalized isogonal map, to be defined in Part IV of this paper, and $Q \cdot Q'$ is the point whose barycentric coordinates are the products of corresponding barycentric coordinates of $Q$ and $Q'$.

\begin{proof}
The proof of I, Theorem 3.8 shows that $\textsf{M}$ is a homothety or translation, since $K$ fixes all points at infinity.  This map takes the triangle $T_{P'}^{-1}(D_0E_0F_0)$, which is inscribed in $\tilde{\C}_O$ by Theorem \ref{thm:K(Q)}, to the triangle $DEF$, which is inscribed in the inconic.  It also takes the center $O=T_{P'}^{-1}K(Q)$ of $\tilde{\C}_O$ to the center $Q$ of $\mathcal{I}$.  Now $Q$ is never the midpoint of a side of triangle $DEF$, because, for instance, $Q=A_0$ would imply $Q=T_P^{-1}(A_0)=D_0$ and then $P'=K^{-1}(Q)=K^{-1}(D_0)=A$, contradicting the fact that $P$ and $P'$ do not lie on the sides of $ABC$.  It follows that $O$ is never the midpoint of a side of $T_{P'}^{-1}(D_0E_0F_0)$, and this implies that $\tilde{\C}_O$ is mapped to $\mathcal{I}$.  This result holds even if the point $O=Q$ is infinite, since then $\textsf{M}$ fixes the tangent at $Q$, i.e., the line at infinity.
\end{proof}

\begin{cor}
\label{cor:tangent}
The tangent to the conic $\tilde{\C}_O$ at the point $T_{P'}^{-1}(D_0)$ is parallel to $BC$, with similar statements for the other vertices of the medial triangle of the anticevian triangle of $Q$.
\end{cor}

\begin{proof}
The tangent to $\tilde{\C}_O$ at $T_{P'}^{-1}(D_0)$ is mapped by $\textsf{M}$ to and therefore parallel to the tangent to $\mathcal{I}$ at $D$, which is $BC$.
\end{proof}

We also need the following proposition in order to complete the proof of the above remark, when $O=D_0$.

\begin{prop}
\label{thm:psi}
Let $\psi_1, \psi_2, \psi_3$ be the mappings of conjugate points on $l_\infty$ which are induced by the polarities corresponding to the inconic $\mathcal{I}$ (center $Q$), the circumconic $\tilde{\C}_O$ (center $O$), and the $9$-point conic $\Nh$ (center $N$).  If $O$ and $Q$ are finite, then $\psi_1=\psi_2=\psi_3$.
\end{prop}

\begin{proof}
If $T$ is any projective collineation, and $\pi$ is the polarity determining a conic $\C$, then the polarity determining $T(\C)$ is $T\pi T^{-1}$.  If $q$ is any non-self-conjugate line for $\pi$, then $T(q)$ is not self-conjugate, since its pole $T(Q)$ does not lie on $T(q)$.  If $\psi$ is the mapping of conjugate points on $q$, then the polar of a point $A$ on $q$ is $\pi(A)=a=Q\psi(A)$.  Hence, the polar of $T(A)$ on $T(q)$ is
\[T\pi T^{-1}(T(A))=T\pi(A)=T(Q\psi(A))=T(Q)T(\psi(A))=T(Q) T\psi T^{-1}(T(A)).\]
This shows that the mapping of conjugate points on $T(q)$ is $T \psi T^{-1}$.  Now apply this to $\tilde{\C}_O$ and $\Nh=K(\tilde{\C}_O)$, with $T=K$ and $q=l_\infty$.  The complement mapping fixes the points on $l_\infty$, so $\psi_3=K\psi_2 K^{-1}=\psi_2$.  Theorem \ref{thm: CtoI} and a similar argument shows that $\psi_1=\psi_2$.
\end{proof}

\begin{cor} The conclusion of Proposition \ref{thm:STparBC} holds if $A=H$ and $O=D_0$, where $S_1$ is defined to be the intersection of the conic $\tilde{\C}_O$ with the line through $A$ parallel to $QD$.
\label{cor:S1}
\end{cor}

\begin{proof}
The line $AO=AT_1$ is a diameter of $\tilde{\C}_O$, so the direction of $AS_1$, which equals the direction of $QD$, is conjugate to the direction of $S_1T_1$.  But the direction of $QD$ is conjugate to the direction of $BC$, since $BC$ is tangent to $\mathcal{I}$ at $D$, and since $\psi_1=\psi_2$.
Therefore, $S_1T_1$ and $BC$ lie on the same point at infinity.
\end{proof}

Proposition \ref{thm:STparBC} and Corollary \ref{cor:S1} will find application in Part IV.  To determine the fixed point of the mapping $\textsf{M}$ in Theorem \ref{thm: CtoI}, we need a lemma.

\begin{lem}
If $P$ does not lie on $\iota(l_\infty)$, the point $\tilde{H}=T_P^{-1}(H)$ is the midpoint of the segment joining $P'$ and $K^{-1}(H)$, and the reflection of the point $Q$ through $O$.
\end{lem}

\begin{proof}
Let $H_1$ be the midpoint of $P'K^{-1}(H)$.  The Euclidean quadrilateral $H_1HQK^{-1}(H)$ is a parallelogram, because $K^{-1}(QH)=P'K^{-1}(H)$ and the segment $H_1K^{-1}(H)$ is therefore congruent and parallel to $QH$.  The intersection of the diagonals is the point $O$, the midpoint of $HK^{-1}(H)$, so that $O$ is also the midpoint of $H_1Q$.  On the other hand, $K(Q)$ is the midpoint of segment $P'Q$, so Theorem \ref{thm:K(Q)} gives that $O=T_{P'}^{-1}K(Q)$ is the midpoint of $T_{P'}^{-1}(P'Q)=Q \tilde{H}$, by (3).  This implies that $H_1=\tilde{H}$.
\end{proof}

\begin{thm}
\label{thm: FixM}
If $P$ is ordinary and does not lie on a median of $ABC$ or on $\iota(l_\infty)$, the fixed point (center) of the map $\textsf{M}=T_P K^{-1} T_{P'}$ is $S=OQ \cdot GV=OQ \cdot O'Q'$.
\end{thm}

\noindent {\bf Remark.} The point $S$ is a generalization of the insimilicenter, since it is the fixed point of the map taking the circumconic to the inconic.  See \cite{ki2}.  \smallskip

\begin{proof}
Assume first that $\textsf{M}$ is a homothety.  The fixed point $S$ of $\textsf{M}$ lies on $OQ$, since the proof of Theorem \ref{thm: CtoI} gives that $\textsf{M}(O)=Q$.  Note that $O \neq Q$, since $T_{P'}(Q)=P' \neq K(Q)$, by I, Theorem 3.7.  We claim that $\textsf{M}(O')=Q'$ also.  We shall prove the symmetric statement $\textsf{M}'(O)=Q$, where $\textsf{M}'=T_{P'}K^{-1}T_P$.  We have that
\[\textsf{M}'(O)=T_{P'} K^{-1} T_P(T_{P'}^{-1}K(Q))=T_{P'}K^{-1} \lambda^{-1}(K(Q)).\]
Now $K(Q)$ is the midpoint of $P'Q$, so $K^{-1} \lambda^{-1}(K(Q))$ is the midpoint of $K^{-1} \lambda^{-1}(P'Q)=K^{-1}(QH) = P'K^{-1}(H)$ (Theorem \ref{thm:lambda}), and therefore coincides with the point $\tilde{H}$, by the lemma.  Therefore, by (3), we have
\[\textsf{M}'(O)=T_{P'}(\tilde{H})=Q.\]
Switching $P$ and $P'$ gives that $\textsf{M}(O')=Q'$, as claimed.  Therefore, $S=OQ \cdot O'Q'$.  Note that the lines $OQ$ and $O'Q'$ are distinct. If they were not distinct, then $O, O', Q, Q'$ would be collinear, and applying $K^{-1}$ would give that $H, H', P', P$ are collinear, which is impossible since these points all lie on the cevian conic $\Cp$.  (Certainly, $H \neq P$, since otherwise $O=T_{P'}^{-1}K(Q)=K(P)=Q'$, forcing $K(Q)=T_{P'}(Q')=Q'=K(P)$ and $P=Q$.  Similarly, $H \neq P'$, so these are four distinct points.)  This shows that $\eta(S)=S$, so $S$ lies on $GV$ and $S=OQ \cdot GV$. \medskip

If $\textsf{M}$ is a translation, then it has no ordinary fixed points, and the same arguments as before give that $OQ \pa O'Q' \pa GV$ and these lines are fixed by $\textsf{M}$.  But then $\textsf{M}$ is a translation along $GV$, so its center is again $S=OQ \cdot GV=OQ \cdot O'Q'$.
\end{proof}

\begin{prop}
\label{thm: Zfixed}
If $P$ does not lie on a median of $ABC$, then $T_PK^{-1}(Z)=Z$.
\end{prop}

\begin{proof}
We first use II, Theorem 4.1, when $P$ and $P'$ are ordinary.  The point $Z$ is defined symmetrically with respect to $P$ and $P'$, since it is the center of the conic $\Cp=\C_{P'}$.  Therefore II, Theorem 4.1 yields $Z=GV \cdot T_{P'}(GV)$, so the last theorem implies that
\[T_PK^{-1}(Z)=T_PK^{-1}(GV) \cdot T_PK^{-1}T_{P'}(GV)=T_P(GV) \cdot GV = Z,\]
since the point $S$ lies on $GV$.  If $P$ lies on $\iota(l_\infty)$, then $T_PK^{-1}(Z)=Z$ follows immediately from II, Theorem 4.3 and the proof of II, Theorem 2.7 (in the case that $P'$ is infinite), since $T_P$ is a translation along $GG_1$ taking $G$ to $G_1=T_P(G)$.  If $P$ is infinite, then $P'$ lies on $\iota(l_\infty)$, in which case $T_PK^{-1}=T_{P'}^{-1}K^{-2}$ by I, Theorem 3.14.  This mapping fixes $Z$ by II, Theorem 4.3.
\end{proof}

\begin{prop}
\label{thm: ZonNh}
If $P$ does not lie on a median of $ABC$, then $Z$ lies on the $9$-point conic $\Nh$ of the quadrangle $ABCH$, and $K^{-1}(Z)$ lies on $\tilde{\C}_O$.
\end{prop}

\begin{proof}
As we remarked in the paragraph just before Theorem \ref{thm:K(Q)}, the point $Z$ lies on $\Npp$.  Theorem \ref{thm:K(Q)} implies that $T_{P'}^{-1}(Z)$ lies on $T_{P'}^{-1}(\Npp)=\tilde{\C}_O$.  By Proposition \ref{thm: Zfixed}, with $P'$ in place of $P$, $T_{P'}^{-1}(Z)=K^{-1}(Z)$.  Since $K^{-1}(Z)$ lies on $\tilde{\C}_O$, the point $Z$ lies on $K(\tilde{\C}_O)=\Nh$.
\end{proof}

\begin{prop}
\label{prop:PhiP}
\begin{enumerate}[a)]
\item The map $\Phi_P=\textsf{M} \ \circ K^{-1}$ satisfies
\[\Phi_P(K(S))=S, \ \Phi_P(N)=Q, \ \textrm{and} \ \ \Phi_P(K(Q'))=T_P(P).\]
The center of the homothety or translation $\Phi_P$ is the common intersection of the lines $GV, NQ,$ and $K(Q')T_P(P)$.
\item Also, $\Phi_P = \Phi_{P'}$, and the maps $T_PK^{-1}$ and $T_{P'}K^{-1}$ commute with each other.
\item $T_{P'}(P')$ lies on the line $OQ$, and $T_P(P)$ lies on $O'Q'$.
\end{enumerate}
\end{prop}

\begin{proof}
a) The map $\Phi_P$ is a homothety or translation by the same argument as in Theorem \ref{thm: CtoI}.  We have $\Phi_P(K(S))=\textsf{M}(S)=S$, while
\begin{align*}
\Phi_P(N)&=T_P \circ K^{-1} \circ T_{P'} \circ K^{-1}(K(O)) = T_P K^{-1} T_{P'}(O)\\
&=T_PK^{-1}(K(Q))=T_P(Q)=Q,
\end{align*}
and
\begin{align*}
\Phi_P(K(Q'))&=T_P \circ K^{-1} \circ T_{P'} \circ K^{-1}(K(Q')) = T_P K^{-1} T_{P'}(Q')\\
&=T_PK^{-1}(Q')=T_P(P).
\end{align*}
It follows that $\Phi_P$ fixes the three lines $GV, NQ$, and $K(Q')T_P(P)$. \smallskip

b) Note that
\begin{align*}
\eta \Phi_P &= \eta \circ T_P \circ K^{-1} \circ T_{P'} \circ K^{-1}\\
&=T_{P'} \circ \eta \circ K^{-1} \circ T_{P'} \circ K^{-1} = T_{P'}  \circ K^{-1} \circ T_{P} \circ K^{-1} \circ \eta\\
&=\Phi_{P'} \eta,
\end{align*}
since $\eta$ and $K^{-1}$ commute.  On the other hand, the center of $\Phi_P$ lies on the line $GV$, which is the line of fixed points for the affine reflection $\eta$.  It follows that $\eta \Phi_P \eta=\Phi_{P'}$ has $l_\infty$ as its axis and $\eta(F)=F$ as its center, if $F$ is the center of $\Phi_P$ (a homology or an elation).  But both maps $\Phi_P$ and $\Phi_{P'}$ share the pair of corresponding points $\{ K(S), S \}$, also lying on $GV$.  Therefore, the two maps must be the same.  (See [co2, pp. 53-54].) \smallskip

c) From a) and b) we have $\Phi_P(K(Q))=\Phi_{P'}(K(Q))=T_{P'}(P')$.  On the other hand, $S$ on $OQ$ implies that $K(S)$ lies on $NK(Q)$, so $K(Q)$ lies on the line $K(S)N$.   Mapping by $\Phi_P$ and using a) shows that $T_{P'}(P')$ lies on $\Phi_P(K(S)N) = SQ=OQ$.
\end{proof}

\noindent {\bf The Generalized Feuerbach Theorem.} \smallskip
{\it If $P$ does not lie on a median of triangle $ABC$, the map
\[\Phi_P=\textsf{M} \circ K^{-1}=T_P \circ K^{-1} \circ T_{P'} \circ K^{-1}\]
takes the $9$-point conic $\Nh$ to the inconic $\mathcal{I}$ and fixes the point $Z$, the center of $\Cp$.  Thus, $Z$ lies on $\mathcal{I}$, and the conics $\Nh$ and $\mathcal{I}$ are tangent to each other at $Z$.  The same map $\Phi_P$ also takes the $9$-point conic $\mathcal{N}_{H'}$ to the inconic $\mathcal{I}'$ which is tangent to the sides of $ABC$ at $D_3, E_3, F_3$.  The point $Z$ is the center of the map $\Phi_P$ (a homology or elation).}  \medskip

\begin{proof}
The mapping $\Phi_P$ takes $\Nh$ to $\mathcal{I}$ by Theorems \ref{thm:NH} and \ref{thm: CtoI}.  Applying Proposition \ref{thm: Zfixed} to the points $P'$ and $P$, we see that $\Phi_P$ fixes $Z$, so $Z$ lies on $\mathcal{I}$ by Proposition \ref{thm: ZonNh}.  First assume $Z$ is an ordinary point.  As a homothety with center $Z$, $\Phi_P$ fixes any line through $Z$, and therefore fixes the tangent $t$ to $\Nh$ at $Z$.  Since tangents map to tangents, $t$ is also the tangent to $\mathcal{I}$ at $Z$, which proves the theorem in this case.  If $Z \in l_\infty$ and $\Nh$ is a parabola, the same argument applies, since the tangent to $\Nh$ at $Z$ is just $l_\infty=\Phi_P(l_\infty)$.  Assume now that $\Nh$ is a hyperbola.  Then $Z$ must be a point on one of the asymptotes $t$ for $\Nh$, which is also the tangent at $Z$.  But $Z$ is on the line $GV$, and by Proposition \ref{prop:PhiP} the center of $\Phi_P$ lies on $GV$.  It follows that if $\Phi_P$ is a translation, it is a translation along $GV$, and therefore fixes the parallel line $t$.  This will prove the assertion if we show that $\Phi_P$ cannot be a homothety when $Z$ is infinite, i.e., it has no ordinary fixed point.  Let $X$ be a fixed point of $\Phi_P$ on the line $GV$.  Writing $\Phi_P=\textsf{M}_1\textsf{M}_2$, with $\textsf{M}_1=T_P \circ K^{-1}$ and $\textsf{M}_2=T_{P'} \circ K^{-1}$, we have by part b) of Proposition \ref{prop:PhiP} that
\[\Phi_P(\textsf{M}_1(X))=\textsf{M}_1(\textsf{M}_2\textsf{M}_1(X))=\textsf{M}_1(X),\]
so $\textsf{M}_1(X)$ is a fixed point of $\Phi_P$ on the line $\textsf{M}_1(GV)=T_P(GV)$.  Assuming $X$ is ordinary, this shows that $\textsf{M}_1(X)=X$, since a nontrivial homothety has a unique ordinary fixed point.  Hence, $X \in GV \cdot T_P(GV)$.  But $Z$ is infinite and $Z=GV \cdot T_P(GV)$, so this is impossible.  Thus, $\Phi_P$ has no ordinary fixed point in this case and its center is $Z$.
\end{proof}

\begin{cor}
\label{thm:genFeuer}
\begin{enumerate}[a)]
\item If $N=K(O)$ is the midpoint of segment $OH$, the center $Z$ of $\Cp$ lies on the line $QN$, and $Z=GV \cdot QN$.
\item The point $K^{-1}(Z)$ lies on the line $OP'$, so that $K^{-1}(Z)=GV \cdot OP'$ lies on $\tilde{\C}_O$.  This point is the center of the anticevian conic $T_P^{-1}(\Cp)$. (See II, Theorem 3.3.)
\item If $Z$ is infinite and the conics $\Nh$ and $\mathcal{I}$ are hyperbolas, the line $QN$ is a common asymptote of $\Nh$ and $\mathcal{I}$.
\end{enumerate}
\end{cor}

\begin{proof}
For part a), Proposition \ref{prop:PhiP} shows that the center $Z$ of $\Phi_P$ lies on the line $NQ$.  For part b), just note that $K^{-1}(NQ)=OP'$ and $K^{-1}(\Nh)=\tilde{\C}_O$.  The last assertion follows from Proposition \ref{thm: Zfixed}.  For part c), the asymptote of $\Nh$ through $Z$ must lie on the center of $\Nh$, which is $N$, and the asymptote of $\mathcal{I}$ through $Z$ must lie on the center of $\mathcal{I}$, which is $Q$.  Therefore, the common asymptote is $QN$.
\end{proof}

\begin{figure}
\[\includegraphics[width=5in]{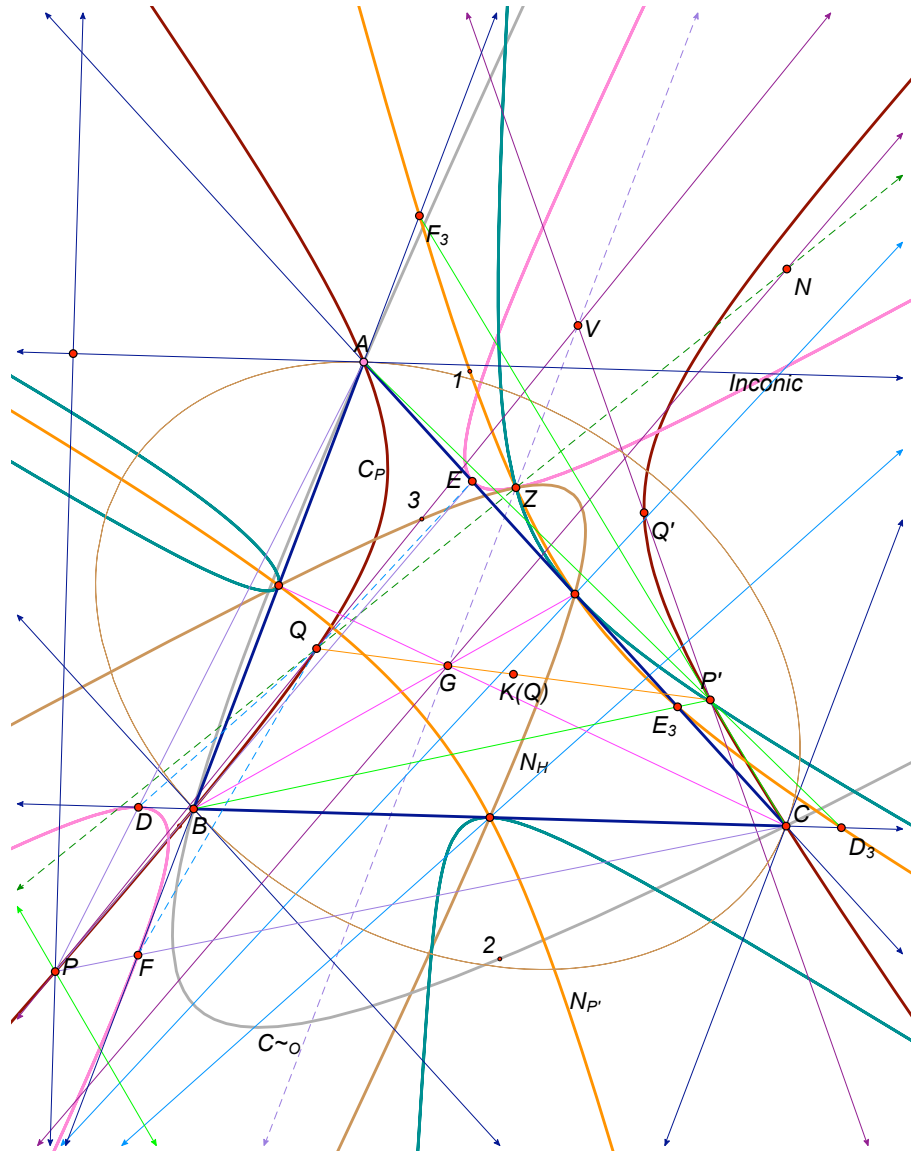}\]
\caption{$\Nh$ (light brown) and $\mathcal{I}$ (pink) tangent at $Z$. (The locus of $Z$ is pictured in teal as $P$ varies on a line.  The ellipse is the Steiner circumellipse.)}
\label{fig:3.2}
\end{figure}

The Generalized Feuerbach Theorem applies to all four points of Theorem \ref{thm:Pabc}, and therefore generalizes the full statement of Feuerbach's theorem in the case that $P$ is the Gergonne point.  Thus, $\Nh$ is tangent to four distinct conics, each of which is tangent to the sides of $ABC$, namely, the inconics corresponding to each of the points $P, P_a, P_b, P_c$.  Figure \ref{fig:3.2} shows the configuration in case $P$ is outside the Steiner circumellipse, in which case $\Nh, \tilde{\C}_O$, and $\mathcal{I}$ are hyperbolas.  The point marked $1$ is a general point on the conic $\Npp$, and the points marked $2$ and $3$ are the images of $1$ on $T_{P'}^{-1}(\Npp)=\tilde{\C}_O$ and on $K(\tilde{\C}_O)=\Nh$, respectively.  As $P$ varies on a line perpendicular to $BC$ in this picture, the locus of the point $Z$ is pictured in teal.  This locus consists of three branches which are each tangent to a side of $ABC$ at its midpoint.  Figure 4 pictures a situation in which $Z$ is infinite.  The point $P$ in this figure was found using the ratios $BD/BC=\frac{15}{16}$ and $BF/AF=\frac{6}{5}$.

\begin{thm}
\label{thm:tildeZ}
The point $\tilde{Z}=\textsf{R}_OK^{-1}(Z)$ is the fourth point of intersection of the conics $\Cp$ and $\tilde{\C}_O$, the other three points being the vertices $A,B,C$.
\end{thm}

\begin{proof}
Theorem \ref{thm:NH} and Proposition \ref{thm: ZonNh} show that $\tilde{Z}=\textsf{R}_OK^{-1}(Z)=\textsf{H}^{-1}(Z)$ lies on $\tilde{\C}_O$.  Since $T_{P'}$ maps $\tilde{\C}_O$ to $\Npp$, we know that the half-turns through the points $O$ and $K(Q)=T_{P'}(O)$ are conjugate by $T_{P'}$, namely:
\[T_{P'} \circ \textsf{R}_O \circ T_{P'}^{-1} = \textsf{R}_{K(Q)}.\]
Therefore, $T_{P'}(\tilde{Z})=T_{P'}\textsf{R}_OK^{-1}(Z)=T_{P'} \textsf{R}_O T_{P'}^{-1}(Z)=\textsf{R}_{K(Q)}(Z)$, the second equality following from Proposition \ref{thm: Zfixed}.  In other words, $Z$ and $T_{P'}(\tilde{Z})$ are opposite points on the conic $\Npp$.  Furthermore, $Z$ lies on $QN$, so $T_{P'}(\tilde{Z})$ lies on the parallel line $l=\textsf{R}_{K(Q)}(QN)$, and since $K(Q)$ is the midpoint of $QP'$, $l$ is the line through $P'$ parallel to $QN$, i.e. $l=OP'=K^{-1}(QN)$.  Hence $T_{P'}(\tilde{Z})$ lies on $OP'$, while Corollary \ref{thm:genFeuer}b) implies that $\tilde{Z}=\textsf{R}_OK^{-1}(Z)$ also lies on $OP'$.  Therefore, $\tilde{Z}, P'$, and $T_{P'}(\tilde{Z})$ are collinear.  Now II, Corollary 2.2b) implies that $\tilde{Z}$ lies on $\C_{P'}=\Cp$.  This shows that $\tilde{Z} \in \Cp \cap \tilde{\C}_O$.
\end{proof}

\begin{figure}
\[\includegraphics[width=5in]{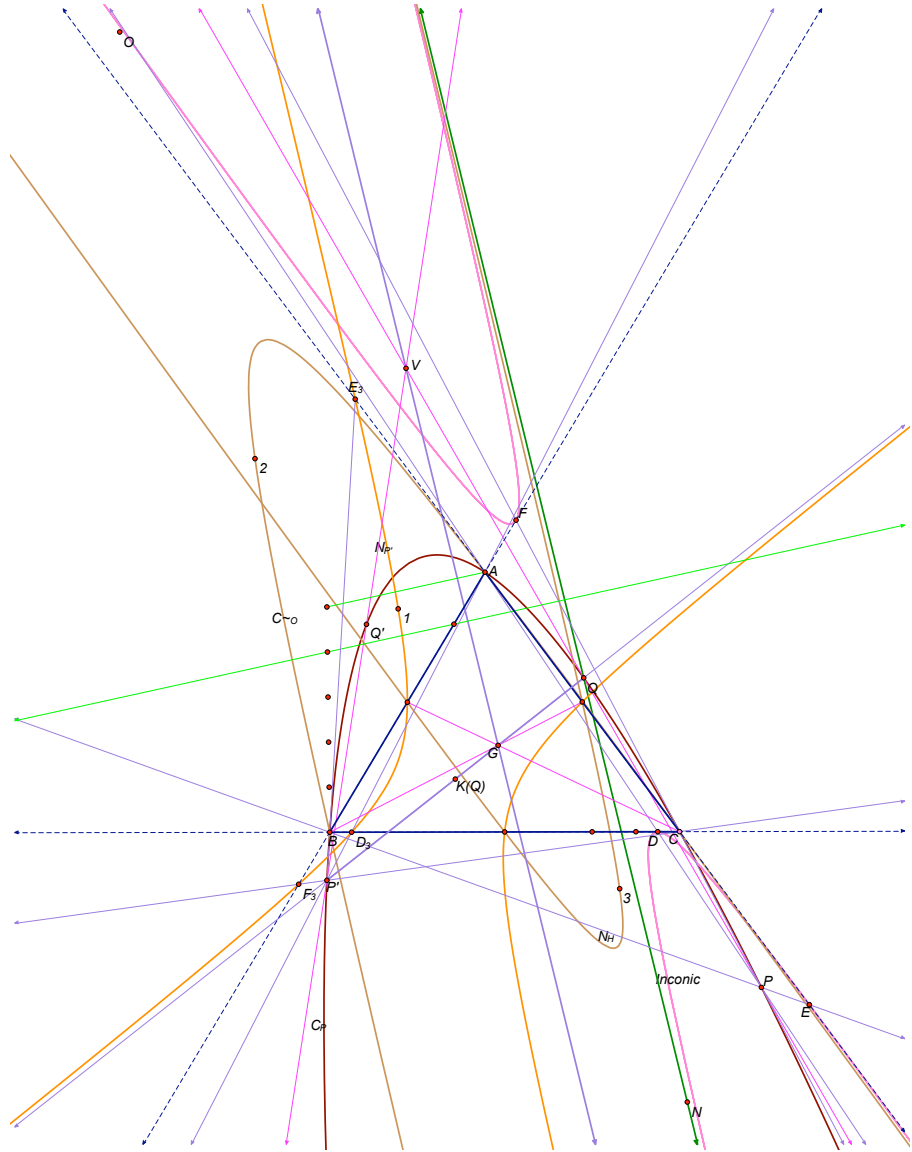}\]
\caption{$Z$ infinite, $\Nh$ (light brown) and $\mathcal{I}$ (pink) tangent to each other with the common asymptote $QN$.  $\Cp$ is the parabola in dark brown.}
\label{fig:3.3}
\end{figure}

\end{section}

\begin{section}{The special case $H=A, O=D_0$.}

We now consider the set of all points $P$ such that $H=A$ and $O=D_0$.  We start with a lemma.

\begin{lem}
Provided the generalized orthocenter $H$ of $P$ is defined, the following are equivalent:
\begin{enumerate}[(a)]
\item $H = A$.
\item $QE = AF$ and $QF = AE$.
\item $F_3$ is collinear with $Q$, $E_0$, and $K(E_3)$.
\item $E_3$ is collinear with $Q$, $F_0$, and $K(F_3)$.
\label{lem:EquivH=A}
\end{enumerate}
\end{lem}
\begin{proof}
We use the fact that $K(E_3)$ is the midpoint of segment $BE$ and $K(F_3)$ is the midpoint of segment $CF$ from I, Corollary 2.2.  Statement (a) holds iff $QE \pa AB$ and $QF \pa AC$, i.e. iff $AFQE$ is a parallelogram, which is equivalent to (b). Suppose (b) holds. Let $X = BE \cdot QF_3$. Then triangles $BXF_3$ and $EXQ$ are congruent since $QE \pa BF_3 = AB$ and $QE = AF = BF_3$. Therefore, $BX = EX$, i.e. $X$ is the midpoint $K(E_3)$ of $BE$, so $Q, F_3$, and $X = K(E_3)$ are collinear. Similarly, $Q, E_3$, and $K(F_3)$ are collinear. This shows (b) $\Rightarrow$ (c), (d).

Next, we show (c) and (d) are equivalent. Suppose (c) holds. Since $P', E_3, B$ are collinear, $Q, K(E_3), E_0$ are collinear and the line $F_3E_0$ is the complement of the line $BE_3$, hence the two lines are parallel and
\begin{equation}
\frac{AF_3}{F_3B} = \frac{AE_0}{E_0E_3}.
\label{eqn:Ratios1}
\end{equation}
Conversely, if the equality holds, then the lines are parallel and $F_3$ lies on the line through $K(E_3)$ parallel to $P'E_3$, i.e. the line $K(P'E_3) = QK(E_3)$, so (c) holds. Similarly, (d) holds if and only if
\begin{equation}
\frac{AE_3}{E_3C} = \frac{AF_0}{F_0F_3}.
\label{eqn:Ratios2}
\end{equation}
A little algebra shows that (\ref{eqn:Ratios1}) holds if and only if (\ref{eqn:Ratios2}) holds.  Using signed distances, and setting $AE_0/E_0E_3 = x$, we have $AE_3/E_3C = (x + 1)/(x - 1)$.  Similarly, if $AF_0/F_0F_3 = y$, then $AF_3/F_3B = (y + 1)/(y - 1)$.  Now  (\ref{eqn:Ratios1}) is equivalent to $x = (y+1)/(y-1)$, which is equivalent to $y = (x+1)/(x-1)$, hence also to (\ref{eqn:Ratios2}).  Thus, (c) is equivalent to (d).  Note that this part of the lemma does not use that $H$ is defined.  \smallskip

Now if (c) or (d) holds, then they both hold.  We will show (b) holds in this case.  By the reasoning in the previous paragraph, we have $F_3Q \pa E_3P'$ and $E_3Q \pa F_3P'$, so $F_3P'E_3Q$ is a parallelogram. Therefore, $F_3Q = P'E_3 = 2\cdot QK(E_3)$, so $F_3K(E_3) = K(E_3)Q$. This implies the triangles $F_3K(E_3)B$ and $QK(E_3)E$ are congruent (SAS), so $AF = BF_3 = QE$. Similarly, $AE = CE_3 = QF$, so (b) holds.
\end{proof}

\begin{thm}
\label{thm:locus}
The locus $\sL_A$ of points $P$ such that $H = A$ is a subset of the conic $\overline{\C}_A$ through $B, C, E_0$, and $F_0$, whose tangent at $B$ is $K^{-1}(AC)$ and whose tangent at $C$ is $K^{-1}(AB)$. Namely, $\mathscr{L}_A = \overline{\C}_A \sm \{B, C, E_0, F_0\}$.
\end{thm}
\begin{proof}
Given $E$ on $AC$ we define $F_3$ as $F_3 = E_0K(E_3) \cdot AB$, and $F$ to be the reflection of $F_3$ in $F_0$.  Then we have the following chain of projectivities:
\[BE \ \barwedge \ E \ \barwedge \ E_3 \ \stackrel{G}{\doublebarwedge} \ K(E_3) \ \stackrel{E_0}{\doublebarwedge} \ F_3 \ \barwedge \ F \ \barwedge \ CF.\]
Then $P = BE \cdot CF$ varies on a line or a conic. We want to show: (a) for a point $P$ thus defined, $H = A$; and (b) if $H = A$ for some $P$, then $P$ arises in this way, i.e. $F_3$ is on $E_0K(E_3)$. Both of these facts follow from the above lemma. \smallskip

Now we list four cases in which $H$ is undefined, namely when $P = B, C, E_0, F_0$.  Let $A_\infty, B_\infty, C_\infty$ represent the points at infinity on the respective lines $BC, AC, AB$.  \smallskip

\begin{enumerate}[1.]
\item For $E = B_\infty = E_3 = K(E_3)$, we have $E_0K(E_3) = AC$ so $F_3 = A, F = B$, and $P = BE \cdot CF = B$.
\item For $E = C$, we have $E_3 = A, K(E_3) = D_0, E_0K(E_3) = D_0E_0 \pa AB, F = F_3 = C_\infty$, so $P = BE\cdot CF = C$.
\item For $E = E_0$, we have $E_3 = E_0$ and $K(E_0)$ is the midpoint of $BE_0$ by I, Corollary 2.2, so $F_3 = B, F = A$, and $P = BE\cdot CF = E_0$.
\item For $E = A$, we have $E_3 = C, K(E_3) = F_0, F_3 = F = F_0$, and $P = BE\cdot CF = F_0$.
\end{enumerate}

\begin{figure}
\[\includegraphics[width=5.5in]{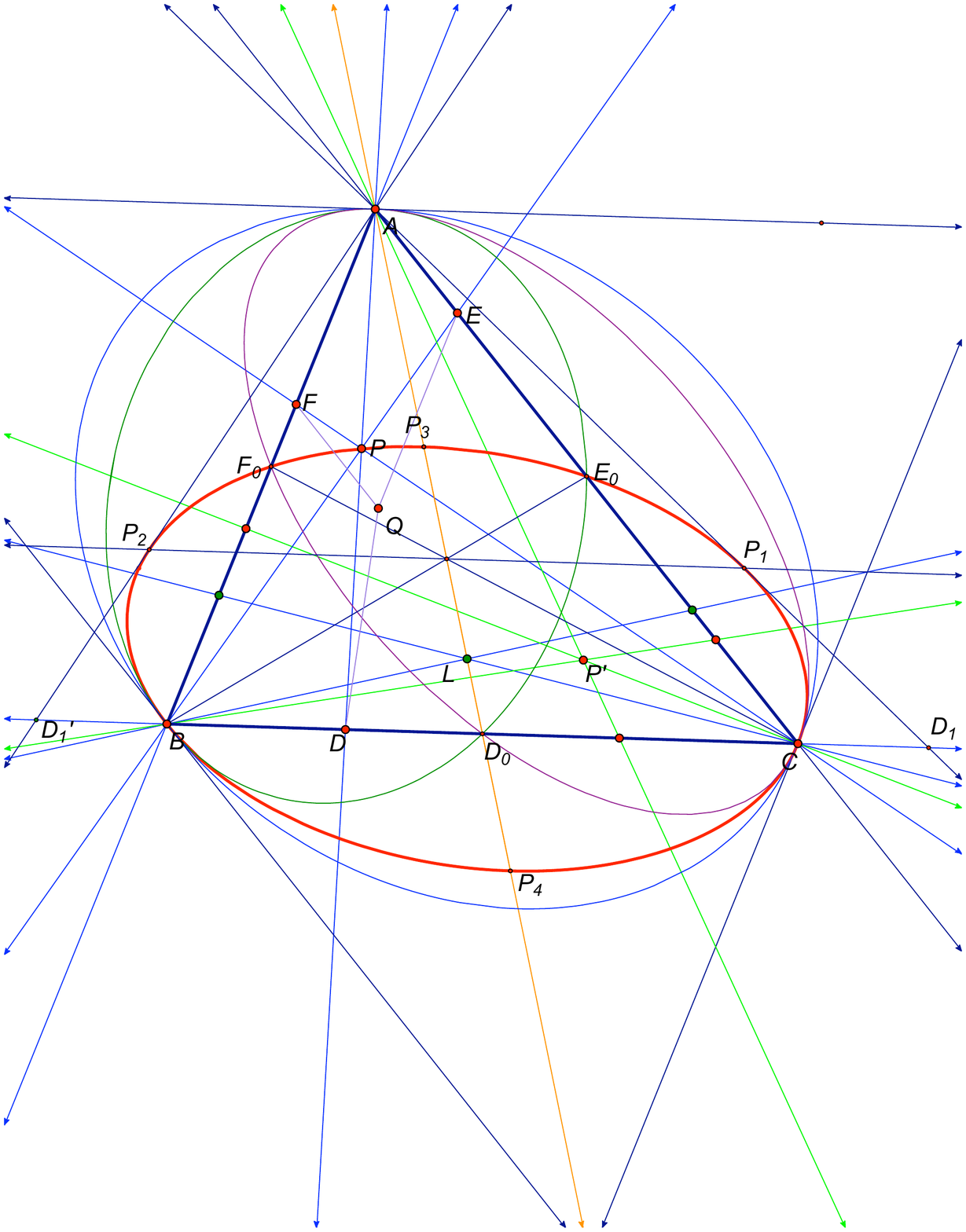}\]
\caption{The conics $\overline{\C}_A$ (red), $\overline{\C}_B$ (purple), $\overline{\C}_C$ (green), and $\iota(l_\infty)$ (blue).}
\label{fig:locus}
\end{figure}

Since the four points $B, C, E_0, F_0$ are not collinear, this shows that the locus of points $P=BE \cdot CF$ is a conic $\overline{\C}_A$ through $B, C, E_0, F_0$.  Moreover, the locus $\mathscr{L}_A$ of points $P$ such that $H = A$ is a subset of $\overline{\C}_A  \sm \{B, C, E_0, F_0\}$. \smallskip

We claim that if $E$ is any point on line $AC$ other than $A, C, E_0$, or $B_\infty$, then $P$ is a point for which $H$ is well-defined. First, $E_3$ is an ordinary point because $E \ne B_\infty$. Second, because $E \ne B_\infty$, the line $E_0K(E_3)$ is not a sideline of $ABC$. The line $E_0K(E_3)$ intersects $AB$ in $A$ if and only if $K(E_3)$ lies on $AC$, which is true only if $E_3 = B_\infty$. The line $E_0K(E_3)$ intersects $AB$ in $B$ iff $K(E_3)$ is on $BE_0$, which holds iff $E_3$ is on $K^{-1}(B)B = BE_0$, and this is the case exactly when $E = E_3 = E_0$.  Furthermore, the line $E_0K(E_3)$ is parallel to $AB$ iff $K(E_3)=D_0$ and $E_3=A$, or $K(E_3) = E_3 = E = C_\infty$, which is not on $AC$. Thus, the line $E_0K(E_3)$ intersects $AB$ in an ordinary point which is not a vertex, so $F_3$ and $F$ are not vertices and $P=BE\cdot CF$ is a point not on the sides of $ABC$. \smallskip

It remains to show that $P$ does not lie on the sides of the anticomplementary triangle of $ABC$. If $P$ is on $K^{-1}(AB)$ then $F=F_3 = C_\infty$, which only happens in the excluded case $E=C$ (see Case 2 above). If $P$ is on $K^{-1}(AC)$ then $E= B_\infty$, which is also excluded. If $P$ is on $K^{-1}(BC)$ then $P'$ is also on $K^{-1}(BC)$ so $Q=K(P')$ is on $BC$.  \smallskip

Suppose $Q$ is on the same side of $D_0$ as $C$. Then $P'$ is on the opposite side of line $AD_0$ from $C$, so it is clear that $CP'$ intersects $AB$ in the point $F_3$ between $A$ and $B$.  If $Q$ is between $D_0$ and $C$, then $F_3$ is between $A$ and $F_0$ (since $F_0, C$ and $K^{-1}(C)$ are collinear), and it is clear that $F_3E_0$ can only intersect $BC$ in a point outside of the segment $D_0C$, on the opposite side of $C$ from $Q$.  But this is a contradiction, since by construction $F_3, E_0$, and $K(E_3)$ are collinear, and $Q=K(P')$ lies on $K(BE_3)=E_0K(E_3)$.  On the other hand, if the betweenness relation $D_0 * C * Q$ holds, then $F_3$ is between $B$ and $F_0$, and it is clear that $F_3E_0$ can only intersect $BC$ on the opposite side of $B$ from $C$.  This contradiction also holds when $P'=Q$ is a point on the line at infinity, since then $F_3=B$, and $B, E_0$ and $Q=A_\infty$ (the point at infinity on $BC$) are not collinear.   A symmetric argument applies if $Q$ is on the same side of $D_0$ as $B$, using the fact that parts (c) and (d) of the lemma are equivalent. Thus, no point $P$ in $\overline{\C}_A  \sm \{B, C, E_0, F_0\}$ lies on a side of $ABC$ or its anticomplementary triangle, and the point $H$ is well-defined; further, $H=A$ for all of these points. \smallskip

Finally, by the above argument, there is only one point $P$ on $\overline{\C}_A$ that is on the line $K^{-1}(AB)$, namely $C$, and there is only one point $P$ on $\overline{\C}_A$ that is on the line $K^{-1}(AC)$, namely $B$, so these two lines are tangents to $\overline{\C}_A$.
\end{proof}

This theorem shows that the locus of points $P$, for which the generalized orthocenter $H$ is a vertex of $ABC$, is the union of the conics $\overline{\C}_A \cup \overline{\C}_B \cup \overline{\C}_C$ minus the vertices and midpoints of the sides. \medskip

In the next proposition and its corollary, we consider the special case in which $H=A$ and $D_3$ is the midpoint of $AP'$.  We will show that, in this case, the map $\textsf{M}$ is a translation.  (See Figure \ref{fig:4.2}.)  We first show that this situation occurs.  \medskip

\begin{lem}
\label{lem:equilateral}
If the equilateral triangle $ABC$ has sides of length $2$, then there is a point $P$ with $AP \cdot BC=D$ and $d(D_0,D)=\sqrt{2}$, such that $D_3$ is the midpoint of the segment $AP'$ and $H=A$.
\end{lem}

\begin{figure}
\[\includegraphics[width=5in]{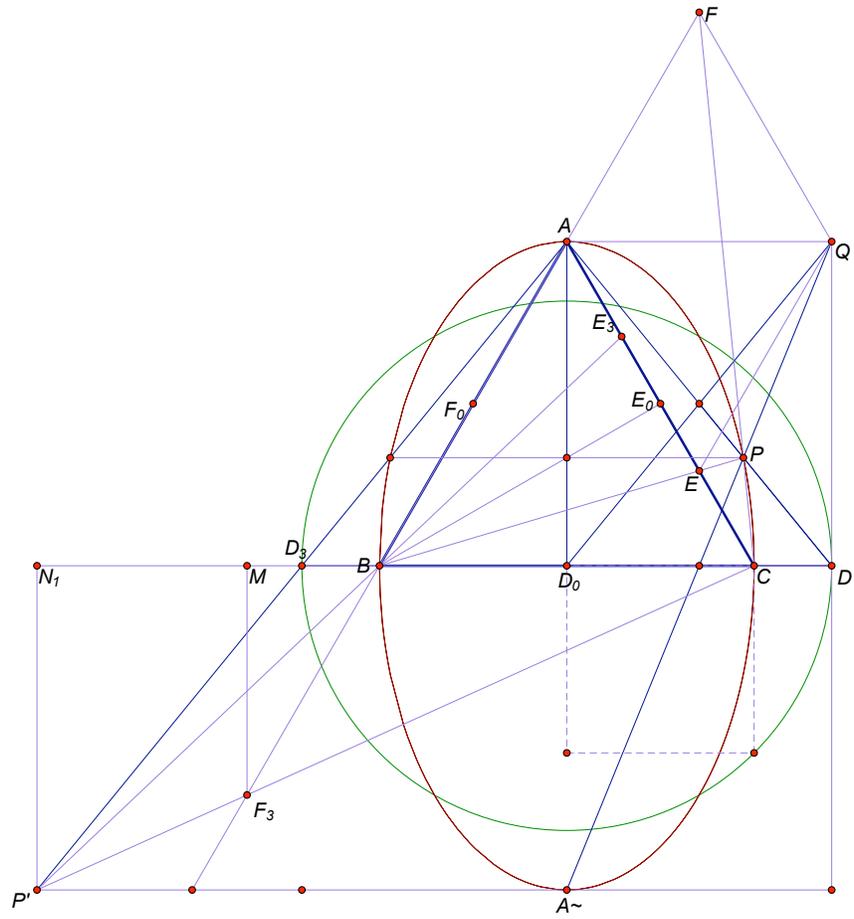}\]
\caption{Proof of Lemma 4.3.}
\label{fig:4.1}
\end{figure}

\begin{proof}
(See Figure \ref{fig:4.1}.) We will construct $P'$ such that $D_3$ is the midpoint of $AP'$ and $H=A$, and then show that $P$ satisfies the hypothesis of the lemma.  The midpoint $D_0$ of $BC$ satisfies $D_0B = D_0C = 1$ and $AD_0 = \sqrt3$. Let the triangle be positioned as in Figure 3.  Let $\tilde A$ be the reflection of $A$ in $D_0$, and let $D$ be a point on $BC$ to the right of $C$ such that $D_0D = \sqrt{2}$. In order to insure that the reflection $D_3$ of $D$ in $D_0$ is the midpoint of $AP'$, take $P'$ on $l=K^{-2}(BC)$ with $P'\tilde A = 2\sqrt2$ and $P'$ to the left of $\tilde A$. Then $Q = K(P')$ is on $K^{-1}(BC)$, to the right of $A$, and $AQ = \sqrt{2}$. Let $E_3$ and $F_3$ be the traces of $P'$ on $AC$ and $BC$, respectively.  \smallskip

We claim $BF_3 = \sqrt{2}$. Let $M$ be the intersection of $BC$ and the line through $F_3$ parallel to $AD_0$. Then triangles $BMF_3$ and $BD_0A$ are similar, so $F_3M = \sqrt{3} \cdot MB$. Let $N_1$ be the intersection of $BC$ and the line through $P'$ parallel to $AD_0$. Triangles $P'N_1C$ and $F_3MC$ are similar, so
\[\frac{F_3M}{MC} = \frac{P'N_1}{N_1C} = \frac{AD_0}{P'\tilde A + 1} = \frac{\sqrt{3}}{2\sqrt{2} + 1}.\]
Therefore,
\[\frac{\sqrt{3}}{2\sqrt{2} + 1} = \frac{F_3M}{MC} = \frac{\sqrt{3} \cdot MB}{MB + 2}\]
which yields that $MB = 1/\sqrt{2}$. Then $BF_3=\sqrt{2}$ is clear from similar triangles. \smallskip

Now, let $F$ be the reflection of $F_3$ in $F_0$ (the midpoint of $AB$). Then $AQF$ is an equilateral triangle because $m(\angle FAQ)=60^\circ$ and $AQ \cong BF_3 \cong AF$, so $\angle AQF \cong \angle AFQ$. Therefore, $QF \pa AC$. It follows that the line through $F_0$ parallel to $QF$ is parallel to $AC$, hence is a midline of triangle $ABC$ and goes through $D_0$. (A similar proof shows that $QE \pa AB$ so the line through $E_0$ parallel to $QE$ goes through $D_0$.) This implies $O = D_0$.  Clearly, $P=AD \cdot CF$ is a point outside the triangle $ABC$, not lying on an extended side of $ABC$ or its anticomplementary triangle, which satisfies the conditions of the lemma.
\end{proof}

The next proposition deals with the general case, and shows that the point $P$ we constructed in the lemma  lies on a line through the centroid $G$ parallel to $BC$.

\begin{prop}
\label{prop:HA}
Assume that $H=A, O=D_0$, and $D_3$ is the midpoint of $AP'$.  Then the conic $\tilde{\C}_O = \iota(l)$, where $l=K^{-1}(AQ)=K^{-2}(BC)$ is the line through the reflection $\tilde{A}$ of $A$ in $O$ parallel to the side $BC$.  The points $O, O', P, P'$ are collinear, with $d(O,P')=3d(O,P)$, and the map $\textsf{M}$ taking $\tilde{\C}_O$ to the inconic $\mathcal{I}$ is a translation.  In this situation, the point $P$ is one of the two points in the intersection $l_G \cap \tilde{\C}_O$, where $l_G$ is the line through the centroid $G$ which is parallel to $BC$.
\end{prop}

\begin{proof}
(See Figure \ref{fig:4.2}.)  Since the midpoint $R_1'$ of segment $AP'$ is $D_3$, lying on $BC$, $P'$ lies on the line $l$ which is the reflection of $K^{-1}(BC)$ (lying on $A$) in the line $BC$.  It is easy to see that this line is $l=K^{-2}(BC)$, and hence $Q=K(P')$ lies on $K^{-1}(BC)$.  From I, Corollary 2.6 we know that the points $D_0, R_1'=D_3$, and $K(Q)$ are collinear.  Since $K(Q)$ is the center of the conic $\Npp$, lying on $D_0$ and $D_3$, $K(Q)$ is the midpoint of segment $D_0D_3$ on $BC$.  Applying the map $T_{P'}^{-1}$ gives that $O=T_{P'}^{-1}(K(Q))$ is the midpoint of $T_{P'}^{-1}(D_3D_0)=AT_{P'}^{-1}(D_0)$.  It follows that $T_{P'}^{-1}(D_0)=\tilde{A}$ is the reflection of $A$ in $O$, so that $\tilde{A} \in \tilde{\C}_O$.  Moreover, $K(A)=O$, so $\tilde{A}=K^{-1}(A)$ lies on $l=K^{-1}(AQ) \pa BC$.   \smallskip

Next we show that $\tilde{\C}_O = \iota(l)$, where the image $\iota(l)$ of $l$ under the isotomic map is a circumconic of $ABC$ (see Lemma 3.4 in Part IV).  It is easy to see that $\iota(\tilde{A})= \tilde{A}$, since $\tilde{A} \in AG$ and $AB\tilde{A}C$ is a parallelogram.  Therefore, both conics $\tilde{\C}_O$ and $\iota(l)$ lie on the $4$ points $A,B,C, \tilde{A}$.  To show they are the same conic, we show they are both tangent to the line $l$ at the point $\tilde{A}$.  From Corollary \ref{cor:tangent} the tangent to $\tilde{\C}_O$ at $\tilde{A}=T_{P'}^{-1}(D_0)$ is parallel to $BC$, and must therefore be the line $l$.  To show that $l$ is tangent to $\iota(l)$, let $L$ be a point on $l \cap \iota(l)$.  Then $\iota(L) \in l \cap \iota(l)$.  If $\iota(L) \neq L$, this would give three distinct points, $L, \iota(L)$, and $\tilde{A}$, lying on the intersection $l \cap \iota(l)$, which is impossible.   Hence, $\iota(L)=L$, giving that $L$ lies on $AG$ and therefore $L=\tilde{A}$.  Hence, $\tilde{A}$ is the only point on $l \cap \iota(l)$, and $l$ is the tangent line.  This shows that $\tilde{\C}_O$ and $\iota(l)$ share $4$ points and the tangent line at $\tilde{A}$, proving that they are indeed the same conic.  \smallskip

From this we conclude that $P=\iota(P')$ lies on $\tilde{\C}_O$.  Hence, $P$ is the fourth point of intersection of the conics $\tilde{\C}_O$ and $\Cp=ABCPQ$.  From Theorem \ref{thm:tildeZ} we deduce that $P= \tilde{Z}=R_OK^{-1}(Z)$, which we showed in the proof of that theorem to be a point on the line $OP'$.  Hence, $P, O, P'$ are collinear, and applying the affine reflection $\eta$ gives that $O'$ lies on the line $PP'$, as well.  Now, $Z$ is the midpoint of $HP=AP$, since $\textsf{H}=K \circ \textsf{R}_O$ is a homothety with center $H=A$ and similarity factor $1/2$.  Since $Z$ lies on $GV$, where $V=PQ \cdot P'Q'$, it is clear that $P$ and $Q$ are on the opposite side of the line $GV$ from $P', Q'$, and $A$.  The relation $K(\tilde{A})=A$ means that $\tilde{A}$ and also $O$ are on the opposite side of $GV$ from $A$ and $O'$.  Also, $J=K^{-1}(Z)=\textsf{R}_O(\tilde{Z})=\textsf{R}_O(P)$ lies on the line $GV$ and on the conic $\tilde{\C}_O$.  This implies that $O$ lies between $J$ and $P$, and applying $\eta$ shows that $O'$ lies between $J$ and $P'$.  Hence, $OO'$ is a subsegment of $PP'$, whose midpoint is exactly $J=K^{-1}(Z)$, since this is the point on $GV$ collinear with $O$ and $O'$.  Now the map $\eta$ preserves distances along lines parallel to $PP'$ (see Part II), so $JO' \cong JO \cong OP \cong O'P'$, implying that $OO'$ is half the length of $PP'$.  Furthermore, segment $QQ'=K(PP')$ is parallel to $PP'$ and half as long.  Hence, $OO' \cong QQ'$, which implies that $OQQ'O'$ is a parallelogram.  Consequently, $OQ \pa O'Q'$, and Theorem \ref{thm: FixM} shows that $\textsf{M}$ is a translation.  Thus, the circumconic $\tilde{\C}_O$ and the inconic $\mathcal{I}$ are congruent in this situation. \smallskip

This argument implies the distance relation $d(O,P') =3d(O,P)$.  \smallskip

The relation $O'Q' \pa OQ$ implies, finally, that $T_P(O'Q') \pa T_P(OQ)$, or $K(Q')P \pa A_0Q = AQ$, since $O'=T_P^{-1}K(Q')$ from Theorem \ref{thm:HO} and $A_0$ is collinear with $A$ and the fixed point $Q$ of $T_P$ by I, Theorem 2.4.  Hence $PG=PQ'=PK(Q')$ is parallel to $AQ$ and $BC$. \smallskip
\end{proof}

\begin{figure}
\[\includegraphics[width=5.5in]{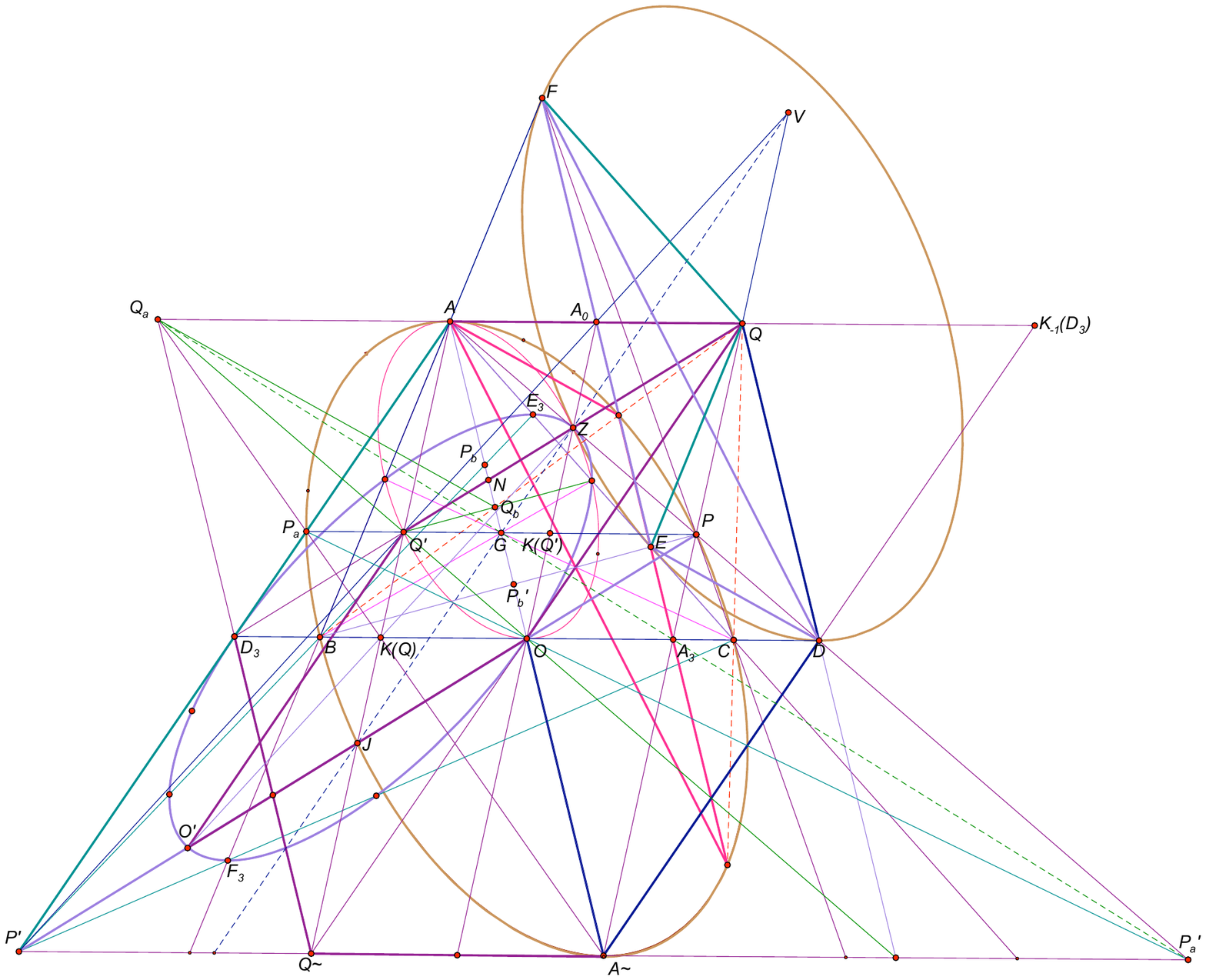}\]
\caption{The case $H=A, O=D_0$.}
\label{fig:4.2}
\end{figure}

There are many interesting relationships in the diagram of Figure \ref{fig:4.2}.  We point out several of these relationships in the following corollary.

\begin{cor}
\label{cor:HArel} Assume the hypotheses of Proposition \ref{prop:HA}.
\begin{enumerate}[a)]
\item If $Q_a$ is the vertex of the anticevian triangle of $Q$ (with respect to $ABC$) opposite the point $A$, then the corresponding point $P_a$ is the second point of intersection of the line $PG$ with $\tilde{\C}_O$.
\item The point $A_3=T_P(D_3)$ is the midpoint of segment $OD$ and $P$ is the centroid of triangle $ODQ$.
\item The ratio $\frac{OD}{OC}=\sqrt{2}$.
\end{enumerate}
\end{cor}

\begin{proof}
The anticevian triangle of $Q$ with respect to $ABC$ is the triangle $T_{P'}^{-1}(ABC)=Q_aQ_bQ_c$.  (See I, Cor. 3.11 and Section 2 above.)  Since $D_3$ is the midpoint of $AP'$, this gives that $T_{P'}^{-1}(D_3)=A$ is the midpoint of $T_{P'}^{-1}(AP')=Q_aQ$.  Therefore, $Q_a$ lies on the line $AQ=K^{-1}(BC)$, so $P_a'=K^{-1}(Q_a)$ lies on the line $l$ and is the reflection of $P'$ in the point $\tilde{A}$.  Thus, the picture for the point $P_a$ is obtained from the picture for $P$ by performing an affine reflection about the line $AG=A\tilde{A}$ in the direction of the line $BC$.  This shows that $P_a$ also lies on the line $PG \pa BC$.  The conic $\tilde{\C}_O$ only depends on $O$, so this reflection takes $\tilde{\C}_O$ to itself.  This proves a). \smallskip

To prove b) we first show that $P$ lies on the line $Q\tilde{A}$.  Note that the segment $K(P'\tilde{A})=AQ$ is half the length of $P'\tilde{A}$, so $P'\tilde{A} \cong Q_aQ$.  Hence, $Q_aQ\tilde{A}P'$ is a parallelogram, so $Q\tilde{A} \cong Q_aP'$.  Suppose that $Q\tilde{A}$ intersects line $PP'$ in a point $X$.  From the fact that $K(Q)$ is the midpoint of $D_3D_0$ we know that $Q$ is the midpoint of $K^{-1}(D_3)A$.  Also, $D_3Q'$ lies on the point $\lambda(A)=\lambda(H)=Q$, by II, Theorem 3.4(b) and Theorem \ref{thm:lambda} of this paper.  It follows that $K^{-1}(D_3), P, P'$ are collinear and $K^{-1}(D_3)QX \sim P'\tilde{A}X$, with similarity ratio $1/2$, since $K^{-1}(D_3)Q$ has half the length of $P'\tilde{A}$.  Hence $d(X, K^{-1}(D_3)) = \frac{1}{2} d(X, P')$.  On the other hand, $d(O,P) = \frac{1}{3} d(O,P')$, whence it follows, since $O$ is halfway between $P'$ and $K^{-1}(D_3)$ on line $BC$, that $d(P, K^{-1}(D_3)) =\frac{1}{2} d(P, P')$.  Therefore, $X=P$ and $P$ lies on $Q\tilde{A}$.  \smallskip

Now, $\textsf{P}=AD_3OQ$ is a parallelogram, since $K(AP')=OQ$, so opposite sides in $AD_3OQ$ are parallel.  Hence, $T_P(\textsf{P})=DA_3A_0Q$ is a parallelogram, whose side $A_3A_0$ lies on the line $EF$.  Applying the dilatation $\textsf{H}=K\textsf{R}_O$ (with center $H=A$) to the collinear points $Q, P, \tilde{A}$ shows that $\textsf{H}(Q), Z$, and $O$ are collinear.  On the other hand, $O=D_0, Z$, and $A_0$ are collinear by I, Corollary 2.6 (since $Z=R$ is the midpoint of $AP$), and $A_0$ lies on $AQ$ by I, Theorem 2.4.  This implies that $A_0=\textsf{H}(Q)=AQ \cdot OZ$ is the midpoint of segment $AQ$, and therefore $A_3$ is the midpoint of segment $OD$.  Since $P$ lies on the line $PG$, $2/3$ of the way from the vertex $Q$ of $ODQ$ to the opposite side $OD$, and lies on the median $QA_3$, it must be the centroid of $ODQ$.  This proves b). \smallskip

To prove c), we apply an affine map taking $ABC$ to an equilateral triangle.  It is clear that such a map preserves all the relationships in Figure 4.  Thus we may assume $ABC$ is an equilateral triangle whose sidelengths are $2$. By Lemma \ref{lem:equilateral} there is a point $P$ for which $AP \cdot BC=D$ with $D_0D=\sqrt{2}, O=D_0$, and $D_3$ the midpoint of $AP'$.  Now Proposition \ref{prop:HA} implies the result, since the equilateral diagram has to map back to one of the two possible diagrams (Figure 4) for the original triangle.
\end{proof}

By Proposition \ref{prop:HA} and Theorem \ref{thm:Pabc} we know that the conic $\overline{\C}_A$ lies on the points $P_1, P_2, P_3, P_4$, where $P_1$ and $P_2=(P_1)_a$ are the points in the intersection $\tilde{\C}_O \cap l_G$ described in Proposition \ref{prop:HA} and Corollary \ref{cor:HArel}, and $P_3=(P_1)_b, P_4=(P_1)_c$.  (See Figure \ref{fig:locus}.)  It can be shown that the equation of the conic $\overline{\C}_A$ in terms of the barycentric coordinates of the point $P=(x,y,z)$ is $xy+xz+yz=x^2$.  Furthermore, the center of $\bar \C_A$ lies on the median $AG$, $6/7$-ths of the way from $A$ to $D_0$.  \medskip

\noindent {\bf Remarks.}
\noindent 1. The polar of $A$ with respect to the conic $\bar \C_A$ is the line $l_G$ through $G$ parallel to $BC$.  This because the quadrangle $BCE_0F_0$ is inscribed in $\bar \C_A$, so its diagonal triangle, whose vertices are $A, G$, and $BC \cdot \l_\infty$, is self-polar. Thus, the polar of $A$ is the line $l_G$. \smallskip

\noindent 2. The two points $P$ in the intersection $\bar \C_A \cap l_G$ have tangents which go through $A$.  This follows from the first remark, since these points lie on the polar $a=l_G$ of $A$ with respect to $\bar \C_A$.  As a result, the points $D$ on $BC$, for which there is a point $P$ on $AD$ satisfying $H = A$, have the property that the ratio of unsigned lengths $DD_0/D_0C \le \sqrt 2$.  This follows from the fact that $\bar \C_A$ is an ellipse: since it is an ellipse for the equilateral triangle, it must be an ellipse for any triangle.  Then the maximal ratio $DD_0/D_0C$ occurs at the tangents to $\bar \C_A$ from $A$; and we showed above that for these two points $P$, $D = AP \cdot BC$ satisfies $DD_0/D_0C = \sqrt 2$.

\end{section}

\noindent Dept. of Mathematics, Carney Hall\\
Boston College\\
140 Commonwealth Ave., Chestnut Hill, Massachusetts, 02467-3806\\
{\it e-mail}: igor.minevich@bc.edu
\bigskip

\noindent Dept. of Mathematical Sciences\\
Indiana University - Purdue University at Indianapolis (IUPUI)\\
402 N. Blackford St., Indianapolis, Indiana, 46202\\
{\it e-mail}: pmorton@math.iupui.edu

\end{document}